\documentclass[11pt]{article}
\usepackage[includehead, includefoot, heightrounded, left=0.8in, top=1.6cm, bottom=2.4cm, right=0.8in]{geometry}
\usepackage[svgnames, dvipsnames, usenames]{xcolor}

\usepackage{tikz}
\usepackage{tgpagella}
\usepackage{graphicx}
\usepackage{amsmath, amssymb, amsthm} 
\usepackage{mathtools} 
\usepackage{bbm} 
\newcommand{\mf}{\mathfrak}
\newcommand{\mc}{\mathcal}

\newcommand{\RR}{\mathbb{R}}

\newcommand{\ZZ}{\mathbb{Z}}

\newcounter{r}
\newcounter{s}
\newcommand\sTableau[1]{
        \foreach \x [count = \c from 1] in {#1} {
		\foreach \y [count = \d from 1] in \x{
			\node at (\d-.5,\c-.5) {\tiny$\y$}; 
			\draw (\d,\c) to (\d,\c-1);
			{\ifnum\d=1
				\draw (0,\c) to (0,\c-1);
				\fi}
			\setcounter{r}{\d}
		}
		{\ifnum\c=1
			\draw (0,0)--(\value{r},0);
			\fi}
		\draw(0,\c) to (\value{r},\c);
		\setcounter{s}{\c}}}

\numberwithin{equation}{section}

\usepackage[pdftitle={A curiously slowly mixing Markov chain}, pdfauthor={Persi Diaconis, Andrew Lin, and Arun Ram},
colorlinks=true,linkcolor=NavyBlue,urlcolor=RoyalBlue,citecolor=PineGreen,bookmarks=true,bookmarksopen=true,bookmarksopenlevel=2,unicode=true,linktocpage]{hyperref}
\usepackage[capitalize, nameinlink]{cleveref}

\usepackage[nottoc,notlot,notlof]{tocbibind}

\newtheorem{theorem}{Theorem}[section]
\newtheorem{corollary}[theorem]{Corollary}
\newtheorem{conjecture}[theorem]{Conjecture}
\newtheorem{proposition}[theorem]{Proposition}
\newtheorem{lemma}[theorem]{Lemma}
\newtheorem{example}[theorem]{Example}

\newtheorem*{remark}{Remark}

\title{A curiously slowly mixing Markov chain}
\author{
\begin{tabular}{c} Persi Diaconis\\ \small Department of Statistics \\ [-3pt] \small Department of Mathematics \\ [-3pt]\small Stanford University\end{tabular}
\begin{tabular}{c} Andrew Lin\\ \small Department of Mathematics \\ [-3pt]\small Stanford University\end{tabular}
\begin{tabular}{c} Arun Ram\\ \small Department of Mathematics \\ [-3pt]\small University of Melbourne\end{tabular}
}
\date{}

\begin{document} 

\maketitle

\begin{abstract}
We study a Markov chain with very different mixing rates depending on how mixing is measured. The chain is the ``Burnside process on the hypercube $C_2^n$.'' Started at the all-zeros state, it mixes in a bounded number of steps, no matter how large $n$ is, in $\ell^1$ and in $\ell^2$. And started at general $x$, it mixes in at most $\log n$ steps in $\ell^1$. But, in $\ell^2$, it takes $\frac{n}{\log n}$ steps for most starting $x$. The $\ell^2$ mixing results follow from an explicit diagonalization of the Markov chain into binomial-coefficient-valued eigenvectors.
\end{abstract}

\tableofcontents

\section{Introduction}

A mainstay of careful analysis of the mixing rates of Markov chains is
\[
    \text{``bound }\ell^1 \text{ by }\ell^2 \text{ and use eigenvalues.''}
\]
While this often works to give sharp rates of convergence in $\ell^1$, even enabling proof of cutoff (a review is found in \cref{backgroundreviewsection}), it can be off if the mixing rates in $\ell^1$ and $\ell^2$ are of different orders. The present paper studies a natural basic example, the Burnside process on $C_2^n$, where we find the mixing rates in $\ell^1$ and $\ell^2$ to be exponentially different.

We begin by describing a general version of the \textit{Burnside process}. Let $\mf{X}$ be a finite set and $G$ a finite group acting on $\mf{X}$. This group action splits $\mf{X}$ into orbits
\[
    \mf{X} = \mc{O}_1 \cup \mc{O}_2 \cup \cdots \cup \mc{O}_Z, \quad \text{where }Z\text{ is the number of orbits.}
\]
We also use the notation $\mc{O}_x$ for the orbit containing $x$. The Burnside process gives a method of choosing an orbit uniformly at random. Examples reviewed in \cref{burnsidesubsection} show that this is an extremely efficient way to generate random integer partitions, P\'olya trees, and many other objects of ``P\'olya theory.'' It proceeds by a Markov chain on $\mf{X}$ run as follows:
\begin{itemize}
    \item From $x \in \mf{X}$, choose $s$ uniformly from the set $G_x = \{s \in G: x^s = x\}$.
    \item Then from that $s \in G$, choose $y$ uniformly from the set $\mf{X}_s = \{y \in \mf{X}: y^s = y\}$.
\end{itemize}
The chance of moving from $x$ to $y$ in one step of the chain is 
\begin{equation}\label{burnsidegenerictransition}
    K(x, y) = \frac{1}{|G_x|} \sum_{s \in G_x \cap G_y} \frac{1}{|\mf{X}_s|}.
\end{equation}
As discussed in \cref{burnsidesubsection}, this is an ergodic, reversible Markov chain on $\mf{X}$ with stationary distribution
\[
    \pi(x) = \frac{1}{Z|\mc{O}_x|}.
\]
Thus, running the chain and simply recording the current orbit gives a Markov chain on orbits with a uniform stationary distribution. 

While experiments show extremely rapid mixing of the Burnside process, this fact has been hard to prove in most settings. A first example, the focus of the present paper, takes $\mf{X} = C_2^n$, the set of binary $n$-tuples, and $G = S_n$, the symmetric group acting by permuting coordinates. Then letting $|x| = \sum_j x_j$ denote the number of ones in $x \in C_2^n$, the orbits can be enumerated as
\begin{equation}\label{binaryburnsideorbits}
    \mc{O}_i = \{x: |x| = i\} \quad \text{for }0 \le i \le n.
\end{equation}
For this example, the two steps of the Burnside process are easy to carry out:
\begin{itemize}
    \item Given $x \in C_2^n$, $G_x$ is the set of permutations which permute the zeros and ones in $x$ among themselves. Thus $G_x \cong S_i \times S_{n-i}$ with $i = |x|$, and it is easy to choose $s \in G_x$ uniformly.
    \item Now given $s \in S_n$, $\mf{X}_s$ is the set of binary $n$-tuples fixed by the permutation of coordinates. To pick an element $y$ of $\mf{X}_s$ at random, we may write $s$ as a product of disjoint cycles, label each cycle with a 0 or 1 according to independent fair coin flips, and assign $y_i = 1$ if and only if the cycle containing $i$ is labeled $1$. It is thus easy to choose $y \in \mf{X}_s$ uniformly.
\end{itemize}
A closed form expression for $K(x, y)$ in the binary case is in \cref{closedformk}. 

For the binary case, a first analysis by Jerrum \cite{jerrum} showed that order $\sqrt{n}$ steps suffice for $\ell^1$ mixing. This was improved by Aldous \cite{aldousfill}, who showed that $\log n$ steps suffice. More precisely, for any $x \in C_2^n$, the \textit{total variation distance} satisfies
\begin{equation}\label{aldousbound}
    ||K_x^\ell - \pi||_{\mathrm{TV}} \le n\left(\frac{1}{2}\right)^\ell \le \left(\frac{1}{2}\right)^c \text{ for } \ell = \log_2n + c.
\end{equation}
This result has been recently refined by Nestoridi and Sly \cite{nestoridisly}, who show (worst-case) total variation cutoff after $\frac{1}{4} \log_2(n)$ steps, with a Gaussian limit profile at window $\sqrt{\log n}$. Independently, Feng and Paguyo \cite{fengpaguyo} obtained an integral kernel representation for the Burnside process on $k$-ary tuples for all $k$ and for generalized ``twisted'' rates. This links the problem to a class of Markov chains studied by Crane and Lalley \cite{cranelalley}, which leads to sharp rates of convergence in total variation.

Of course, the starting state can matter. In \cite{diaconiszhong}, it is shown that starting at the all-zeros state $\underline{0}$, just a bounded number of steps suffice:
\begin{equation}\label{diaconiszhongbound}
    \frac{1}{4}\left(\frac{1}{4}\right)^\ell \le ||K_{\underline{0}}^\ell - \pi||_{\mathrm{TV}} \le 4\left(\frac{1}{4}\right)^\ell.
\end{equation}
This result is proved by the ``bound $\ell^1$ by $\ell^2$'' approach. Here, the $\ell^2$ or \textit{chi-square distance} after $\ell$ steps is
\[
    \chi_x^2(\ell) = \sum_y \frac{(K^\ell(x, y) - \pi(y))^2}{\pi(y)} = \left|\left|\frac{K_x^\ell}{\pi} - 1\right|\right|_2^2.
\]  
As illustrated in \cref{diaconiszhongbound}, often the bound
\[
    4||K_x^\ell - \pi||_{\mathrm{TV}}^2 \le \chi_x^2(\ell)
\]
is fairly sharp (\cref{backgroundreviewsection} provides bounds in the other direction). However, this requires the $\ell^1$ and $\ell^2$ mixing times to be of the same order.

\subsection{Main results}

Our first main result shows that for the binary Burnside process, the $\ell^1$ and $\ell^2$ mixing times can have very different orders. For most starting states, order $\frac{n}{\log n}$ steps are required to make $\chi_x^2(\ell)$ small, which is exponentially slower than \cref{aldousbound}. To state the result, define the average chi-square distance as
\[
    \chi_{\text{avg}}^2(\ell) = \sum_x \pi(x) \chi_x^2(\ell).
\]  
We use the notation $\Theta_c(f(n))$ for a function bounded above and below by $C_1 f(n)$ and $C_2f(n)$, with constants $C_1, C_2$ depending on $c$.

\begin{theorem}\label{chisquare}
For the binary Burnside process on $C_2^n$, we have the following:
\begin{enumerate}
    \item $\displaystyle\chi_{\text{avg}}^2(\ell) = \sum_{k=1}^{\lfloor n/2 \rfloor} \binom{n}{2k} \left(\frac{1}{2^{4k}} \binom{2k}{k}^2\right)^{2\ell}$.
    \item Take any $\varepsilon > 0$. If $\ell \le (1-\varepsilon) \frac{\log 2}{2} \frac{n}{\log n}$, then $\chi_{\text{avg}}^2(\ell) \to \infty$ as $n \to \infty$. Meanwhile, if $\ell \ge (1+\varepsilon)\frac{\log 2}{2} \frac{n}{\log n}$, then $\chi_{\text{avg}}^2(\ell) \to 0$ as $n \to \infty$. 
    
    Furthermore, these results can be upgraded to per-state bounds: there is some sequence of starting states $x^{(n)} \in C_2^n$ such that $\chi_{x^{(n)}}^2(\ell) \to \infty$ even after $(1 - \varepsilon)\frac{\log 2}{2} \frac{n}{\log n}$ steps, while all starting states $x$ have $\chi_x^2(\ell) \to 0$ after just $(1 + \varepsilon)\frac{\log 2}{2} \frac{n}{\log n}$ steps. 
    \item Let $x^{(n)} \in C_2^n$ be any sequence of states such that $cn \le |x^{(n)}| \le (1 - c)n$ for some $c \in (0, 1)$. Then $\chi_{x^{(n)}}^2(\ell) \to \infty$ for $\ell = \Theta_c(\frac{n}{\log n})$.
    \item In particular, let $y^{(n)} \in C_2^n$ be the states with $\lfloor \frac{n}{2} \rfloor$ zeros followed by $\lceil \frac{n}{2} \rceil$ ones. Then $\chi_{y^{(n)}}^2(\ell) \to \infty$ for $\ell \le (1-\varepsilon) \frac{\log 2}{2} \frac{n}{\log n}$, matching the constant in part (2).
\end{enumerate}
\end{theorem}

In words, both the average chi-square distance and the chi-square distance started at the ``half-zeros, half-ones state'' $y^{(n)}$ have an $\ell^2$ cutoff at $\ell = \frac{\log 2}{2} \frac{n}{\log n}$. 

The proofs of these results rely on an explicit diagonalization of the chain, which we describe now.

\begin{theorem}\label{eigenvectors}
Let $K(x, y)$ be the transition matrix of the binary Burnside process on $C_2^n$. We have the following:
\begin{enumerate}
    \item The eigenvalues of $K$ are $0$ and
    \[
        \beta_k = \frac{1}{2^{4k}} \binom{2k}{k}^2, \quad 0 \le k \le \left\lfloor \frac{n}{2} \right\rfloor.
    \]
    \item The eigenvalue multiplicity of $0$ is $2^{n-1}$. The eigenvalue multiplicity of $\beta_k$ is $\binom{n}{2k}$ for all $0 \le k \le \lfloor \frac{n}{2} \rfloor$.
    \item For $0 \le k \le \lfloor\frac{n}{2}\rfloor$, a basis of eigenvectors for the $\beta_k$-eigenspace is 
    \[
        \left\{f_S(x) = (-1)^{|x_S|} \binom{2k}{|x_S|}: |S| \subset [n], |S| = 2k\right\},
    \]
    where $|x_S|$ denotes the number of ones of $x \in C_2^n$ among the coordinate set $S$.
\end{enumerate}
\end{theorem}

\begin{remark}
Alas, these eigenvectors $f_S$ are not orthogonal. We describe formulas in \cref{eigenvectorsection} for the inner products $\langle f_S, f_T \rangle$ for any subsets $S, T$, but the usual route of expressing $\chi_x^2(\ell)$ requires an orthonormal eigenbasis. We perform this analysis, using Schur--Weyl duality, in a companion paper \cite{diaconislinram2}.
\end{remark}

\subsection{Outline}

We conclude with a summary of the rest of this paper. \cref{backgroundreviewsection} below gives background on some required analytic tools (\cref{analyticsubsection}), as well as some additional uses for the eigenvectors (\cref{eigenvectorsusesubsection}). It also gives a survey of examples where $\ell^1$ and $\ell^2$ rates are the same and different (\cref{l1vsl2subsection}) and a brief review of the Burnside process (\cref{burnsidesubsection}).

Properties of the transition matrix $K$ are developed in \cref{propertiessection}. We point to a curious, but important, feature: for any subset $S \subseteq [n]$, the chain on $C_2^n$ lumped to $S$ is \textit{precisely} the Burnside process on $C_2^{|S|}$. \cref{eigenvectors} is then proved in \cref{eigenvectorsection}, and \cref{chisquare} is proved in \cref{chisquareboundsection}.

Finally, \cref{relatedchainssection} compares some of the properties of the binary Burnside process, including eigenvalue multiplicities and duality, to other chains in which these properties manifest. It then shows that some of the magical properties of the binary Burnside chain also hold for the chain on $C_k^n$ for $k \ge 2$, and it proposes some conjectures which would generalize the results that we have established in the binary case.

\section*{Acknowledgements}

This paper is dedicated to our friend Svante Janson on the occasion of his 70th birthday. Over the years, we have tried (and failed) to interest him in mixing questions; we sure could use his help on this problem. We thank him for a lifetime of brilliant work and his help throughout that lifetime.

We also thank Evita Nestoridi, Allan Sly, Amol Aggarwal, Yuval Peres, Michael Howes, Jonathan Hermon, Lucas Teyssier, Laurent Bartholdi, Christoph Koutschan, Richard Stanley, Sourav Chatterjee, Ferenc Balogh, Ivan Feng, and J. E. Paguyo for helpful discussions and ideas. Finally, we are also grateful to two anonymous referees for insightful comments and corrections.

AI tools were used to correct and refine the ideas in the proof of \cref{innerproductwz}, but all text in the manuscript is written fully by the authors.

P.D.\ was partially supported by the NSF under Grant No.\ DGE-1954042, and A.L.\ was partially supported by the NSF under Grant No.\ DGE-2146755.

\section{Background}\label{backgroundreviewsection}

This section contains needed background and a literature review. \cref{analyticsubsection} reviews the analytic background for bounding $\ell^1$ and $\ell^2$ distances using eigenvalues and eigenvectors, and \cref{eigenvectorsusesubsection} provides further applications of those eigenvectors in the case of the binary Burnside chain. \cref{l1vsl2subsection} gives examples of Markov chains where the $\ell^1$ and $\ell^2$ distances are well-understood enough to give useful comparisons. Finally, \cref{burnsidesubsection} discusses previous literature on the Burnside process.

\subsection{Analytic background}\label{analyticsubsection}

An exceptional text for mixing time results is the book by Levin and Peres \cite{levinperes}. Chapter 12 of their book contains the basics for bounding $\ell^1$ and $\ell^2$ distances using eigenvalues. The comprehensive text of Saloff-Coste \cite{saloffcoste} develops analytic tools more deeply.

Let $\mf{X}$ be a finite set and $K(x, y)$ a Markov transition matrix with state space $\mf{X}$ and stationary distribution $\pi(x)$. Throughout this section, assume that $(K, \pi)$ is ergodic and reversible. For any $1 \le p < \infty$, let $\ell^p(\pi) = \{f: \mf{X} \to \RR\}$ denote the function space with norm
\[
    ||f||_p^p = \sum_x |f(x)|^p \pi(x).
\]
Reversibility implies that $Kf(x) = \sum_y K(x, y) f(y)$ is self-adjoint as a map $\ell^2 \to \ell^2$; that is, for any functions $f, g$, we have $\langle Kf, g \rangle = \langle f, Kg \rangle$, where $\langle f, g \rangle = \sum_x f(x) g(x) \pi(x)$. Thus, the spectral theorem shows that $K$ has an orthonormal set of eigenvectors and corresponding eigenvalues $f_i, \beta_i$ with $Kf_i(x) = \beta_i f_i(x)$ for all $i$. As usual, we reorder so that $1 = \beta_0 > \beta_1 \ge \beta_2 \ge \cdots \ge \beta_{|\mf{X}| - 1} > -1$. Started at a state $x$, the $\ell^1$ or \textit{total variation} distance 
\[
    ||K_x^\ell - \pi||_{\mathrm{TV}} = \frac{1}{2} \sum_y |K^\ell(x, y) - \pi(y)|
\]
and the $\ell^2$ or \textit{chi-square} distance 
\[
    \chi_x^2(\ell) = \sum_y \frac{(K^\ell(x, y) - \pi(y))^2}{\pi(y)} = \left|\left|\frac{K_x^\ell}{\pi} - 1\right|\right|_2^2
\]
can be bounded as 
\begin{equation}\label{l1vsl2bounds}
4||K_x^\ell - \pi||_{\mathrm{TV}}^2 \le \chi_x^2(\ell) = \sum_{i=1}^{|\mf{X}| - 1} f_i^2(x) \beta_i^{2\ell} \le \frac{1}{\pi(x)} \beta_\ast^{2\ell},
\end{equation}
where $\beta_\ast = \max(|\beta_1|, |\beta_{|\mf{X}| - 1}|)$ is the second absolute eigenvalue.

If $f$ is an eigenfunction for $K$ with eigenvalue $\beta \ne 1$, 
\begin{align*}
    |\beta^\ell f(x)| = \left|K^\ell f(x)\right| &= \left|\sum_y K^\ell(x, y) f(y) \right| \\
    &= \left|\sum_y K^\ell(x, y) f(y) - \pi(y) f(y) \right| \\
    &\le \sum_y |K^\ell(x, y) - \pi(y)| f^\ast,
\end{align*}
where $f^\ast = \max_y |f(y)|$. Choosing $x^\ast$ so that $|f(x^\ast)| = f^\ast$ then yields
\[
    |\beta^\ell| \le \sum_y |K^\ell(x, y) - \pi(y)| = 2 ||K_{x^\ast}^\ell - \pi||_{\mathrm{TV}}.
\]
Combining this with $\sum_{i=0}^{|\mf{X}| - 1} f_i^2(x) = \frac{1}{\pi(x)}$ and \cref{l1vsl2bounds} yields the following:

\begin{proposition}\label{twosidedl1l2bound}
Let $(K, \pi)$ be a reversible Markov chain with second absolute eigenvalue $\beta_\ast$, and let $f$ be an eigenvector for $\beta_\ast$. Suppose $x_\ast \in \mf{X}$ satisfies $|f(x^\ast)| = \max_x |f(x)|$. Then 
\[
    4||K_{x^\ast}^\ell - \pi||_{\mathrm{TV}}^2 \le \chi_{x^\ast}^2(\ell) \le \frac{2}{\pi(x^\ast)} ||K_{x^\ast}^{2\ell} - \pi||_{\mathrm{TV}}.
\]  
\end{proposition}

This is one of several results showing roughly that if $||K^\ell - \pi||_{\mathrm{TV}}$ is close to zero in order $\ell^\ast$ steps, then $\chi^2(\ell)$ is close to zero in order $\ell^\ast + \log \pi(x^\ast)$ steps. (See \cite[Chapter 12]{levinperes}.) Here is an example application to the binary Burnside process:

\begin{corollary}\label{l2byl1corollary}
For the binary Burnside process on $C_2^n$ for any $n \ge 3$ and any $x \in C_2^n$,
\[
    \chi_x^2(\ell) \le \frac{2}{\pi(x)} ||K_x^{2\ell} - \pi||_{\mathrm{TV}}.
\]
\end{corollary}
\begin{proof}
By \cref{eigenvectors}, the second absolute eigenvalue is $\beta_1 = \frac{1}{4}$. Furthermore, for each $1 \le i < j \le n$, we have the corresponding $\beta_1$-eigenfunction 
\begin{equation}\label{beta1eigenfunction}
    f_{\{i,j\}}(x) = (-1)^{x_i + x_j} \binom{2}{x_i + x_j} = \begin{cases} 1 & \text{ if }(x_i, x_j) = (0, 0) \text{ or }(1, 1), \\ -2 & \text{ if }(x_i, x_j) = (0, 1) \text{ or }(1, 0). \end{cases} \\
\end{equation}
Thus if $x$ is not the all-zeros $\underline{0}$ or all-ones $\underline{1}$ state, there is at least one choice of $i, j$ with $|f_{\{i,j\}}(x)| = 2$ and the result follows. And if $x = \underline{0}$ or $\underline{1}$, then the sum of all $\binom{n}{2}$ $f_{\{i,j\}}$s is a $\beta_1$-eigenfunction achieving its maximum magnitude at $x = \underline{0}$ and $\underline{1}$, so again we may apply the previous result.
\end{proof}

\begin{example}\label{l2byl1example}
Using \cref{aldousbound} to bound the right-hand side of this corollary, and noting that for any $x$, 
\[
    \frac{1}{\pi(x)} = (n+1) \binom{n}{|x|} \le 2n^{1/2} 2^n,
\]  
we have that for any starting state,
\[
    \chi_x^2(\ell) \le 4n^{3/2} 2^n \left(\frac{1}{2}\right)^{2\ell}.
\]
This shows that $\frac{n}{2} + c \log n$ steps suffice for $\ell^2$ mixing for every starting state $x$, and \cref{chisquare} shows that this is indeed close to best possible. 

On the other hand, the bound is sometimes much tighter. If $|x|$ or $n - |x|$ is of constant size (meaning that all but a constant number of coordinates in the starting state are all $0$s or all $1$s), then the corollary states that $\chi_x^2(\ell)$ is at most $\left(\frac{1}{2}\right)^{2\ell}$ times a polynomial in $n$, and therefore just $c \log n$ steps will suffice for such states.
\end{example}

\begin{remark}
It is of course important to point out that there are many other approaches to proving rates of convergence. For example, Aldous' bound (\cref{aldousbound}) is proved by coupling, and strong stationary times have also proved useful; see \cite[Chapters 5 and 6]{levinperes} for definitions and references. The direct use of Cauchy-Schwarz to bound $\ell^1$ by $\ell^2$ can also be refined, as is done in \cite{teyssierlimitprofile} and \cite{nestoridilimitprofile} in the analysis of limit profiles of reversible Markov chains. Other tools include functional inequalities such as the Nash \cite{nashineq}, Harnack \cite{harnack}, and log-Sobolev \cite{logsobolev} inequalities, as well as the emerging work on spectral independence \cite{spectralindependence}. 

Despite all of this, ``bound $\ell^1$ by $\ell^2$ and use eigenvalues'' is still basic and useful. In particular, it must be mentioned that one of the advantages of $\ell^2$ bounds is the availability of comparison theory \cite{comparisontheory}. If one has a rate in $\ell^2$, then it is often possible to get good rates of closely related chains, such as perturbations of the kernel or more drastic variations. (For example, the walk on permutations generated by a single transposition and a single $n$-cycle gets sharp bounds via comparison with random transpositions.) This kind of robustness does not seem to be available for other methods of proof. Additionally, the chi-square distance also serves to bound other widely used discrepancies, such as the (squared) Hellinger distance and relative entropy
\[
    H^2_x(\ell) = \sum_y \left(\sqrt{K^\ell(x, y)} - \sqrt{\pi(y)}\right)^2, \quad \text{KL}_x(\ell) = \sum_y K^\ell(x, y) \log \frac{K^\ell(x, y)}{\pi(y)},
\]
via $H^2_x(\ell) \le \text{KL}_x(\ell) \le \chi_x^2(\ell)$. For proofs and further comparisons with other probability metrics, see \cite{choosemetric}.
\end{remark}

\subsection{Some uses for the eigenvectors}\label{eigenvectorsusesubsection}

The bounds for $\chi_x^2(\ell)$ above show an example calculation where ``the eigenvectors can be used for something.'' When studying convergence of a Markov chain, it can also be informative to see how certain key features of that chain converge. We show now in several examples how the explicit form of our eigenvectors $f_S$ can be useful for such questions. 

One natural statistic on binary strings is the number of \textit{alternations} $T(x)$ (that is, the count of adjacent differing coordinates); for example, $T(011011) = 3$. The celebrated work of Tversky and Kahneman \cite{misperceptionchance} on misperceptions of chance applies to the way we view the ``hot hand'' in basketball games \cite{hothand}, the effect of weather on arthritis pain \cite{weathereffect}, and many other examples. All of these misperceptions happen because most people think that a random sequence should have very many alternations and no long runs of zeros or ones. 

Under a fair coin-tossing model for a uniformly random binary sequence of length $n$, all alternations are independent and thus $T(x)$ has mean $\frac{n-1}{2}$ and standard deviation $\frac{\sqrt{n-1}}{2}$. In particular, when $n$ is large, $T(x)$ normalized by this mean and standard deviation has an approximately normal distribution. If we instead observe a coin-tossing sequence where the probability of heads $p$ is unknown and has uniform prior on $[0, 1]$, then the resulting sequence is exactly distributed as $\pi(x)$ for the binary Burnside process (see the second half of \cref{othermarkovchains} for some further discussion on a related Gibbs sampler chain). Alternations at different coordinates are now no longer independent, and the following discussion shows that $T(x)$ under $\pi(x)$ has a completely different behavior to that in the uniform case.

\begin{proposition}
Let $T_n$ be the random variable $T(x)$ under $\pi_n$ on $C_2^n$. Then $\frac{T_n}{n-1}$ converges in distribution to $2U(1-U)$, where $U$ is uniform on $[0, 1]$.
\end{proposition}

A histogram of sampled alternation counts (along with the limiting density) is shown in \cref{fig:alternations}. Notice in particular that even though $T(x)$ can be as large as $n-1$, it typically takes values below $\frac{n}{2}$; the following proof provides an explanation for this.

\begin{figure}
    \centering
    \includegraphics[width=\linewidth]{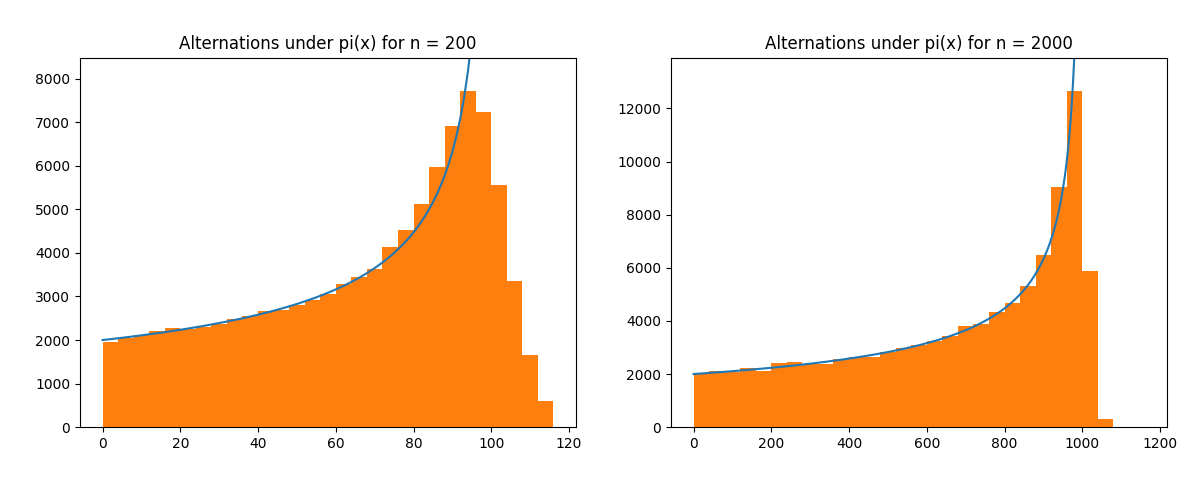}
    \caption{Alternation count histograms for $100000$ binary strings sampled under $\pi_n(x)$ for $n = 200$ and $n = 2000$. The smooth curve corresponds to the limiting density $\frac{1}{\sqrt{1 - 2x}}$ for a random variable distributed as $2U(1-U)$ for $U$ uniform.}\label{fig:alternations}
\end{figure}

\begin{proof}
The measure $\pi_n$ on $C_2^n$ restricted to a subset of $m < n$ coordinates is exactly the measure $\pi_m$ on $C_2^m$; this can be shown by direct calculation or by interpreting binary strings as draws from P\'olya's urn. (In fact, more is true: \cref{restricttocoordinates} shows that a similar statement also holds for the binary Burnside process.) Thus, by consistency of finite-dimensional distributions, we may extend the $\pi_n$s to a measure on infinite binary sequences. Recall that each $\pi_n$ is uniform on the number of ones and that the coordinates are exchangeable under $\pi_n$. So by de Finetti's theorem, this limiting infinite measure can be described as ``pick $p$ uniformly on $[0, 1]$ and flip a $p$-coin independently for each coordinate.''

Thus if we condition on the tail $\sigma$-field and let $p$ be the conditional probability that $x_i = 1$, then the $x_i$s form an iid string of $p$-coin-tosses. Letting $Y_i = 1\{x_i \ne x_{i+1}\}$ (and still conditioning on the tail), the $Y_i$s are each Bernoulli with parameter $2p(1-p)$ and form a $1$-dependent sequence of random variables; in particular we have that $\frac{T_n}{n-1} = \frac{1}{n-1} \sum_{i=1}^{n-1} Y_i$ converges in probability to $2p(1-p)$ by the weak law of large numbers, uniformly in $p$. Thus unconditionally we must have $\frac{T_n}{n-1}$ converge to $2U(1-U)$ for $U$ uniform on $[0, 1]$, as desired.
\end{proof}

In particular, while $\frac{T_n}{n-1}$ converges in probability to $\frac{1}{2}$ under the uniform measure on binary strings, it instead converges to some nondegenerate random variable with mean $\frac{1}{3}$ under $\pi_n$. We now show convergence of $\mathbb{E}[T(X_n)]$ to its stationary average using our $\beta_1$-eigenfunctions:

\begin{example}
Consider alternations in the binary Burnside process on $C_2^n$. We have (using the formula \cref{beta1eigenfunction}) that $1\{x_i \ne x_{i+1}\} = \frac{1 - f_{\{i,i+1\}}(x)}{3}$ for any $1 \le i \le n-1$. Therefore summing over all $i$, the total number of alternations is
\[
    T(x) = \sum_{i=1}^{n-1}1\{x_i \ne x_{i+1}\} = \frac{n-1}{3} - \frac{1}{3} \sum_{i=1}^{n-1} f_{\{i,i+1\}}(x).
\]
Each $f_{\{i,i+1\}}(x)$ is an eigenfunction of $K$ of eigenvalue $\frac{1}{4}$. Thus if $X_0 = x, X_1, \cdots, X_\ell, \cdots$ records the location of the binary Burnside process started at $x$, we have
\begin{align*}
    \mathbb{E}[T(X_\ell) | X_0 = x] &= \mathbb{E}[K^\ell T(x)] \\
    &= \frac{n-1}{3} - \frac{1}{3}\sum_{i=1}^{n-1} \mathbb{E}[K^\ell f_{\{i,i+1\}}(x)] \\
    &= \frac{n-1}{3} - \frac{1}{3 \cdot 4^\ell} \sum_{i=1}^{n-1} \mathbb{E}[f_{\{i,i+1\}}(x)] \\
    &= \frac{n-1}{3} - \frac{1}{4^\ell}\left(\frac{n-1}{3} - T(x)\right);
\end{align*}
in particular, the expected number of alternations is always close to its mean under $\pi_n$ after just $\log_4 n + c$ steps from any starting state.
\end{example}

Similar calculations also yield the same type of exponential decay for other ``pairwise-coordinate'' statistics, such as the covariance of the number of ones between distinct coordinate sets of $x$. Indeed, for any disjoint sets $S, T \subseteq [n]$ (remembering that $|x_S|$ denotes the number of ones in $x$ among the set $S$),
\begin{align*}
    \text{Cov}(|x_S|, |x_T|) &= \sum_{i \in S, j \in T} \text{Cov}(x_i, x_j) \\
    &= \sum_{i \in S, j \in T} \mathbb{E}\left[\left(x_i - \frac{1}{2}\right) \left(x_j - \frac{1}{2}\right) \right] \\
    &= \sum_{i \in S, j \in T} \frac{1}{4} \left(1\{x_i = x_j\} - 1\{x_i \ne x_j\} \right) \\
    &= \sum_{i \in S, j \in T} \frac{1}{6} \left(f_{\{i,j\}}(x) + \frac{1}{2}\right),
\end{align*}
so again the deviation from the mean $\mathbb{E}[\text{Cov}(|x_S|, |x_T|)] = \frac{1}{12}|S||T|$ will decay exponentially in the number of steps taken.

\medskip

Since the number of alternations $T(X_\ell)$ does not concentrate around its mean, it is also informative to compute its variance. For this, we again expand out in terms of eigenvectors:
\begin{align*}
    \text{Var}(T(X_\ell) | X_0 = x) &= \text{Var}\left(\frac{1}{3} \sum_{i=1}^{n-1} f_{\{i, i+1\}}(X_\ell) \,\middle|\, X_0 = x\right) \\
    &= \frac{1}{9}\, \text{Cov}\left(\sum_{i=1}^{n-1}f_{\{i, i+1\}}(X_\ell), \sum_{j=1}^{n-1}f_{\{j, j+1\}}(X_\ell) \,\middle|\, X_0 = x\right) \\
    &= \frac{1}{9}\sum_{i,j = 1}^{n-1} \Bigg(\mathbb{E}\left[f_{\{i, i+1\}}(X_\ell)f_{\{j, j+1\}}(X_\ell) \,\middle|\, X_0 = x \right] \\
    &\hspace{2cm}- \mathbb{E}\left[f_{\{i, i+1\}}(X_\ell) \,\middle|\, X_0 = x \right]\mathbb{E}\left[f_{\{j, j+1\}}(X_\ell) \,\middle|\, X_0 = x \right]\Bigg).
\end{align*}
The latter term decays exponentially because each $f_{\{i, i+1\}}$ is an eigenvector of eigenvalue $\frac{1}{4}$, but for the former term we need to first rewrite $f_{\{i, i+1\}} f_{\{j, j+1\}}$ as a linear combination of eigenvectors, which looks different for each of the cases $i=j$, $|i-j| = 1$, and $|i-j| > 1$. For distinct indices $a, b, c, d$, we have
\[
    f_{\{a,b\}} f_{\{c, d\}} = \frac{18}{35} f_{\{a,b,c,d\}} - \frac{2}{7} (f_{\{a, b\}} + f_{\{c, d\}}) + \frac{3}{14} (f_{\{a,c\}} + f_{\{a,d\}} + f_{\{b,c\}} + f_{\{b,d\}}) + \frac{1}{5} f_{\varnothing},
\]
\[
    f_{\{a, b\}} f_{\{a, c\}} = \frac{3}{2} f_{\{b, c\}} - \frac{1}{2} (f_{\{a, b\}} + f_{\{a, c\}}) + \frac{1}{2} f_{\varnothing},
\]
\[
    f_{\{a, b\}}^2 = -f_{\{a, b\}} +2 f_{\varnothing}.
\]
Plugging in these formulas and using the identities \textbf{(a)} $\mathbb{E}\left[f_{\varnothing}(X_\ell) \,\middle|\, X_0 = x \right] = 1$, \textbf{(b)} $\mathbb{E}\left[f_{\{a, b\}}(X_\ell) \,\middle|\, X_0 = x \right] = \left(\frac{1}{4}\right)^\ell f_{\{a, b\}}(x)$, and \textbf{(c)} $\mathbb{E}\left[f_{\{a, b, c, d\}}(X_\ell) \,\middle|\, X_0 = x \right] = \left(\frac{9}{64}\right)^\ell f_{\{a, b, c, d\}}(x)$ yields an expression for $\text{Var}(T(X_\ell) | X_0 = x)$ in terms of only exponential factors in $\ell$ and eigenfunctions evaluated at $x$. If we only keep the terms corresponding to $f_{\varnothing}$ (since all other terms are exponentially decaying), we find the variance of $T(x)$ under the stationary distribution $\pi_n$. Thus for all $n \ge 2$, $\text{Var}(T(X_\ell) | X_0 = x)$ is always asymptotic to
\[
    \text{Var}_{\pi_n}(T(x)) = \frac{1}{9}\left((n-1) \cdot 2 + (2n-4) \cdot \frac{1}{2} + (n^2 - 5n + 6) \cdot \frac{1}{5} \right) = \frac{1}{45}(n^2 + 10n - 14).
\]
In particular, $\text{Var}_{\pi_n}\left(\frac{T(x)}{n-1}\right)$ does converge to $\frac{1}{45}$, the variance of $2U(1-U)$ for $U$ uniform on $[0, 1]$, as $n \to \infty$.

\begin{remark}
In \cref{eigenvectors}, we claim that $f_S$ is an eigenvector of $K$ for any subset $S$ of even size. In fact, the proof of this result (in \cref{eigenvectorsection}) shows that the formula $f_S(x) = (-1)^{|x_S|} \binom{k}{|x_S|}$ also yields an eigenvector of eigenvalue $0$ for any $S$ of odd size $k$, and that the collection of all such eigenvectors is a basis for the $0$-eigenspace. Other applications of our eigenvectors (i.e. expanding other functions as linear combinations involving binomial coefficients) may find this basis useful. For example, letting $|X_\ell|$ be the number of ones after $\ell$ steps, we find for any starting $x$ and all $\ell \ge 1$ (using $|x| = \frac{n - \sum_{i=1}^n f_{\{i\}}(x)}{2}$) that
\[
    \mathbb{E}\left[|X_\ell|\right] = \frac{n}{2}, \quad \text{Var}(|X_\ell|) = \frac{n(n+2)}{12} \cdot \left(1 - \frac{1}{4^\ell}\right) + \left(|x| - \frac{n}{2}\right)^2 \cdot \frac{1}{4^\ell}.
\]
That is, the number of ones has the right mean after just one step, but it takes order $\log n$ steps to get the variance right.
\end{remark}

\subsection{$\ell^1$ versus $\ell^2$ examples}\label{l1vsl2subsection}

By now, the literature of carefully worked examples is so large that a serious survey would require a book-length effort; see \cite{saloffcoste} for a start. In general, changing the metric of convergence can result in drastically different results, and Gibbs--Su \cite{choosemetric} provides a general survey. We content ourselves with examples drawn from our own work. 

The first examples of sharp mixing time analyses are found in Aldous \cite{aldous1983} and Diaconis--Shahshahani \cite{randomtranspositions}. Random transpositions on the symmetric group $S_n$ was found to have a cutoff at the mixing time of $\frac{1}{2} n \log n + cn$, both in $\ell^1$ and $\ell^2$. Similar results were found for simple random walk on $C_2^n$, simple random walk on the $n$-cycle, and the Bernoulli-Laplace urn. All results were proved by ``bound $\ell^1$ by $\ell^2$ and use eigenvalues.''

Mixing occurs at the same order, but with different constants for the cutoff in $\ell^1$ and $\ell^2$, for random walk on certain expanding graphs called Ramanujan graphs; see Lubetzky--Peres \cite{cutoffexpanders} for precise statements. There are a host of other results where mixing occurs at the same order (up to constants) in $\ell^1$ and $\ell^2$; most of the examples that use the previously mentioned ``comparison theory'' fall into this class. For example, random walk on the symmetric group $S_n$ based on choosing either the transposition $(1, 2)$ or the $n$-cycle $(1, 2, \cdots, n)$ with probability $\frac{1}{2}$ each (for $n$ odd) mixes in order $n^3 \log n$ for both $\ell^1$ and $\ell^2$.

These classical examples have had considerable development. For example, Olesker-Taylor, Teyssier, and Th\'evenin \cite{comparablel1l2conjugacy} show that any random walk supported on any conjugacy class in $S_n$ (such as random $3$-cycles or random $n$-cycles) have comparable $\ell^1$ and $\ell^2$ rates of convergence. In a sustained development, Guralnick, Larsen, Liebeck, Shalev, and Tiep (in various combinations) have shown the same fact for walks supported on conjugacy classes for finite groups of Lie type. A convenient reference is \cite{characterlietype}, detailing a large development.

A different part of the spectrum concerns finite abelian and nilpotent groups. A series of papers by Hermon, Olesker-Taylor, and Huang have close to complete results -- see \cite{hermontaylor, hermonhuang}.

The above is a pale summary of a rich literature, but this is not the time or place for further details.

There are also more refined ``limit profile'' results where ``bound $\ell^1$ by $\ell^2$'' does not provide sufficiently refined estimates. For example, the limit shape results for random walk on the hypercube \cite{nearestneighbor}, riffle shuffles \cite{riffleshuffle}, and random transpositions \cite{teyssierlimitprofile} all require more detailed analysis. 

In the other camp, there are some cases where $\ell^1$ and $\ell^2$ rates are simply different. The easiest example is lazy simple random walk on the complete graph on $n$ vertices, which has bounded $\ell^1$ mixing time but needs order $\log n$ in $\ell^2$. For a more striking example, Peres and Revelle \cite{lamplighter} study simple random walk on the lamplighter group with underlying graph the $n$-cycle $C_n$. They show that order $n^2$ steps are necessary and sufficient for $\ell^1$ convergence, while order $n^3$ steps are necessary and sufficient for $\ell^2$ convergence. The present paper offers an even more extreme example where the $\ell^1$ and $\ell^2$ bounds are exponentially different.

\begin{remark}
As an additional sidenote, continuous-time analogs of discrete-time Markov chains may also have notably different mixing time behavior. While Chen and Saloff-Coste \cite{chensaloffcoste} prove that lazy discrete-time Markov chains exhibit total variation cutoff if and only if the associated continuous-time Markov processes do, Hermon and Peres \cite{hermonperes} show that this is no longer true when using the closely-related metric of \textit{separation distance}. Turning to $\ell^2$ mixing, Saloff-Coste and Z\'u\~niga \cite{zuniga} show that cutoff times occur at different orders in discrete and continuous time for some conjugacy-class random walks on the symmetric and alternating group, and they remark that this occurs due to the effect of a ``very large number of very small eigenvalues.'' Our calculations in this paper for the binary Burnside process show that high eigenvalue multiplicity can also manifest in differences between $\ell^1$ and $\ell^2$, even when restricted only to discrete time.
\end{remark}

\subsection{The Burnside process}\label{burnsidesubsection}

The Burnside process was introduced by Jerrum \cite{jerrum} and studied by Goldberg and Jerrum \cite{goldbergjerrum}. Any such process may be thought of as a bipartite random walk on $\mf{X} \cup G$ with $x \in \mf{X}$ and $s \in G$ connected if $x^s = x$, and it is in this context that the stationary distribution $\pi(x) = \frac{1}{Z|\mc{O}_x|}$ was derived (\cite{jerrum}, Theorem 2.1). The original motivation was computational complexity, and Goldberg and Jerrum produced examples where the Markov chain requires exponentially many steps to converge. 

More practical applications later appeared: Diaconis--Tung \cite{diaconistung} and Diaconis--Howes \cite{contingency} use the chain as an extremely effective algorithm for generating uniform partitions of $n$ (of size roughly $10^8$) and large contingency tables. Diaconis--Zhong \cite{burnsideimportance} uses the Burnside process to generate random conjugacy classes in the group $U(n, q)$, where for example for $n = 40, q = 2$, the orbit space has order roughly $2^{400}$. And Bartholdi--Diaconis \cite{polyatrees} describes an algorithm for using the chain to generate large uniform unlabeled trees and compare various statistics with the corresponding labeled trees. In all of these examples, extensive empirical testing indicates that the chain converges after just $100$ steps. However, \textit{no} explicit rates have been proven for any of these examples.

The binary Burnside process has been carefully studied as a first example towards proving such rates, and the papers of Jerrum, Aldous, Diaconis--Zhong, Nestoridi--Sly, and Feng--Paguyo in the introduction have all contributed to this. The present paper shows that even this simple case has unknown corners. Good results for the Burnside process on conjugacy classes of CA groups are developed by Rahmani \cite{CAgroups}, and good results for set partitions can be found in the work of Paguyo \cite{paguyo}. Most recently, a sharp analysis of mixing time and limit profile for the Burnside process on certain Sylow double cosets of $S_n$ was obtained by Howes \cite{howes}.

\section{Properties of the chain}\label{propertiessection}

This section develops symmetry and lumping properties of the binary Burnside process. Throughout the rest of the paper, we may write $K_n$ for $K$ for the sake of clarity.

To begin, here is a closed form for the transition matrix:

\begin{proposition}\label{closedformk}
Let $K$ denote the transition matrix for the binary Burnside process. Fix $x, y \in C_2^n$, and for $a, b \in \{0, 1\}$, let $n_{ab}$ be the number of coordinates $i$ where $x_i = a$ and $y_i = b$ (so $n_{00} + n_{01} + n_{10} + n_{11} = n$). Then 
\[
    K(x, y) = \frac{\binom{2n_{00}}{n_{00}} \binom{2n_{01}}{n_{01}} \binom{2n_{10}}{n_{10}} \binom{2n_{11}}{n_{11}}}{4^n\binom{n_{00} + n_{01}}{n_{00}}\binom{n_{10} + n_{11}}{n_{10}}}.
\]
\end{proposition}
\begin{proof}
The permutations that fix $x$ can be described by $G_x = S_{n_{00} + n_{01}} \times S_{n_{10} + n_{11}}$, in which we permute the indices where $x_i = 0$ and also the indices where $x_i = 1$. Similarly, $G_x \cap G_y = S_{n_{00}} \times S_{n_{01}} \times S_{n_{10}} \times S_{n_{11}}$ is the set of all permutations whose cycles are each contained entirely within each type of coordinate. If $\sigma = \sigma_{00} \times \sigma_{01} \times \sigma_{10} \times \sigma_{11} \in G_x \cap G_y$, then
\[
    \frac{1}{|\mf{X}_g|} = \left(\frac{1}{2}\right)^{C(\sigma_{00}) + C(\sigma_{01}) + C(\sigma_{10}) + C(\sigma_{11})},
\]
where $C(\tau)$ denotes the number of cycles in the permutation $\tau$, since a binary $n$-tuple is fixed by $\sigma$ if and only if it is constant (either all $0$ or all $1$) on each cycle. Therefore, by plugging into \cref{burnsidegenerictransition},
\begin{align*}
    K(x, y) &= \frac{1}{(n_{00} + n_{01})! (n_{10} + n_{11})!} \sum_{\sigma_{00}, \sigma_{01}, \sigma_{10}, \sigma_{11}}\left(\frac{1}{2}\right)^{C(\sigma_{00}) + C(\sigma_{01}) + C(\sigma_{10}) + C(\sigma_{11})} \\
    &= \frac{1}{(n_{00} + n_{01})! (n_{10} + n_{11})!} \sum_{\sigma_{00}} \left(\frac{1}{2}\right)^{C(\sigma_{00})} \sum_{\sigma_{01}} \left(\frac{1}{2}\right)^{C(\sigma_{01})} \sum_{\sigma_{10}} \left(\frac{1}{2}\right)^{C(\sigma_{10})} \sum_{\sigma_{11}} \left(\frac{1}{2}\right)^{C(\sigma_{11})}.
\end{align*}
Now because the generating function for permutation cycle count is given by
\[
    C_n(x) = \sum_{\sigma \in S_n} x^{C(\sigma)} = x(x+1)\cdots(x+n-1)
\]
(for example by induction, since there are $(n-1)$ ways to insert the number $n$ into an existing cycle and $1$ way to add a new cycle with just $n$), we have that 
\[
    \sum_{\sigma_{00}} \left(\frac{1}{2}\right)^{C(\sigma_{00})} = \frac{1}{2} \left(\frac{1}{2} + 1\right) \cdots \left(\frac{1}{2} + n_{00} - 1\right) = \frac{(2n_{00}-1)!!}{2^{n_{00}}} = \frac{(2n_{00})!}{4^{n_{00}} n_{00}!}
\]  
and similarly for the other terms. Thus
\begin{align*}
K(x, y) &= \frac{1}{(n_{00} + n_{01})! (n_{10} + n_{11})!}  \frac{(2n_{00})!}{4^{n_{00}} n_{00}!} \frac{(2n_{01})!}{4^{n_{01}} n_{10}!} \frac{(2n_{10})!}{4^{n_{10}} n_{10}!} \frac{(2n_{11})!}{4^{n_{11}} n_{11}!} \\
&= \frac{1}{4^n} \frac{n_{00}! n_{01}! n_{10}! n_{11}!}{(n_{00} + n_{01})! (n_{10} + n_{11})!} \frac{(2n_{00})!}{n_{00}!^2} \frac{(2n_{01})!}{n_{01}!^2} \frac{(2n_{10})!}{n_{10}!^2} \frac{(2n_{11})!}{n_{11}!^2},
\end{align*}
which rearranges to the desired result.
\end{proof}

\begin{corollary}\label{Kmatrixsymmetries}
The transition matrix $K(x, y)$ satisfies
\[
    K(x, y) = K(x, \overline{y}) = K(\overline{x}, \overline{y}) = K(\sigma(x), \sigma(y)),
\]  
where $\overline{x}$ is the binary $n$-tuple obtained from $x$ by flipping all bits and $\sigma \in S_n$ is any permutation.
\end{corollary}
\begin{proof}
All equalities follow from observing that the quantity in \cref{closedformk} is left invariant under permutation of coordinates or negation of either $x$ or $y$.
\end{proof}

The first two equalities in \cref{Kmatrixsymmetries} describe an additional symmetry of the chain which will be helpful in obtaining eigenvectors, while the last equality may be viewed as a consequence of the lumping of the binary Burnside process to its orbits (since $|x| = |\sigma(x)|$ for all $x \in C_2^n$ and $\sigma \in S_n$). Similar analysis for other Markov chains with various symmetries, towards comparing the Metropolis algorithm to other random walks on graphs, can be found in \cite{boydsymmetry}.

To state the next (crucial) feature of $K(x, y)$, recall that a function of a Markov chain need not be a Markov chain. For the (general) Burnside process, as in the introduction, \cite[Section 3]{boseeinstein} shows that the chain ``lumped to orbits'' remains a Markov chain with a uniform stationary distribution. Background on lumping, in particular Dynkin's criterion, can be found in \cite{kemenysnell} or \cite{lumping}. Its application to the binary Burnside process underlies the results for $||K_{\underline{0}}^\ell - \pi||_{\mathrm{TV}}$ and $\chi_{\underline{0}}^2(\ell)$ stated in the introduction, using that $\underline{0}$ is in its own orbit. This next result shows that a very different set of lumpings also remains valid:

\begin{proposition}\label{restricttocoordinates}
The restriction of the Burnside process on $(C_2^n, S_n)$ to any $m \le n$ of its coordinates is also a Markov chain, and its transition probabilities are exactly given by the Burnside process on $(C_2^m, S_m)$.
\end{proposition}

\begin{proof}
Given any permutation $\sigma \in S_n$, we may write it uniquely in cycle notation by cyclically moving the largest element in each cycle to the beginning and then sorting those largest elements from smallest to largest. For example, the one-line permutation $26375841$ becomes $(3)(5)(74)(8126)$. Observe that if we remove the parentheses from this expression, then a new cycle begins at every left-to-right record (that is, every largest number starting from the left), so we can read off the cycle notation simply from the sequence of numbers.

But now for any $t \in \{1, 2, \cdots, n\}$, if we remove the numbers $1$ through $t$ in our cycle notation, this does not change whether each of the numbers $t+1$ through $n$ is a record or not. Thus regardless of the order in which $1$ through $t$ appear, we will get the same permutation on $t+1$ through $n$ after erasing $1$ through $t$. In fact, any two elements of $\{t+1, \cdots, n\}$ end up in the same cycle if and only if they were in the same cycle before the erasure. So if we sample a uniformly random permutation from $S_n$ and then erase the numbers $1$ through $t$ in the cycle notation described above, then the result is a uniformly random permutation from $S_{n-t}$ (on the remaining numbers).

Now consider what happens in one step of the Burnside process. For clarity, first consider the case where we restrict to the last $m$ coordinates. If we are at some state $x$ in $C_2^n$ at the start of a step, then the coordinate set can be partitioned into the locations of $0$s and $1$s in $x$; call these sets $L_0, L_1$. The process specifies that we first pick a uniform group element fixing $x$, which is a uniform permutation from $S_{L_0} \times S_{L_1}$, then write this permutation in cycle notation, and then uniformly assign to each cycle either $0$ or $1$. But if we erase the numbers $1, \cdots, n-m$ from our cycle notation, by the logic in the previous paragraph, we get a uniform permutation on $S_{L_0 \cap [n-m+1, n]}$ and on $S_{L_1 \cap [n-m+1, n]}$, hence a uniform permutation among all choices of $S_{L_0 \cap [n-m+1, n]} \times S_{L_1 \cap [n-m+1, n]}$. Furthermore, each of the remaining cycles is still uniformly assigned either $0$ or $1$ because the ``being in the same cycle'' property is preserved. Therefore, we see that Dynkin's criterion for lumping holds here, since the transition probabilities into each potential orbit (indexed by the values of the tuple on the last $m$ indices) are the same regardless of the values of the first $n-m$ indices, and they are exactly those given by the Burnside process on $C_2^m$.

This argument also works if we restrict to any other subset of the indices instead of the last $m$ coordinates; the only modification is that instead of sorting the cycles in purely increasing order, we use an ordering of $[n]$ such that all elements of the subset are ordered after all other elements. This completes the proof.
\end{proof}

Our strategy for writing down explicit eigenvectors of $K$ in the next section will combine lumping down to the binary Burnside process for smaller $n$ and lumping to the orbits $\mc{O}_i$ of \cref{binaryburnsideorbits}.

\section{Eigenvalues and eigenvectors; proof of \cref{eigenvectors}}\label{eigenvectorsection}

The determination of the eigenvalues and eigenvectors of the binary Burnside process on $C_2^n$ depends on the explicit diagonalization of the orbit chain $K^{\text{lumped}}$ on $\{0, 1, \cdots, n\}$. To be precise, this is the chain which follows the dynamics of the binary Burnside process but then only records the orbit $\mc{O}_i$ of the state (which we identify with the integer $i$). An explicit form for $K^{\text{lumped}}(i, j)$ (which is not needed here) is found in \cite[Eqn.(3.1)-(3.3)]{boseeinstein}. The diagonalization we require is recorded here:

\begin{proposition}[{\cite[Theorem 2]{diaconiszhong}}]\label{lumpedeigenvalues}
The eigenvectors of the Markov chain $K^{\text{lumped}}_n$ are the discrete Chebyshev polynomials $T_n^m$ on $\{0, 1, \cdots, n\}$. The nonzero eigenvalues are $\beta_k = \frac{1}{2^{4k}} \binom{2k}{k}^2$ for $k = 0, 1, \cdots, \lfloor \frac{n}{2} \rfloor$, corresponding to the even-degree Chebyshev polynomials $T_n^{2k}$, respectively. All remaining eigenvalues are zero, corresponding to the odd-degree Chebyshev polynomials $T_n^m$ for odd $m \le n$.
\end{proposition}

Here, the discrete Chebyshev polynomials are the orthogonal polynomials for the uniform distribution, where $T_n^m$ is the polynomial of degree $m$ (viewed as a vector by evaluating at the points $\{0, 1, \cdots, n\}$). This result is enough to prove \cite[Theorem 1]{diaconiszhong}, which states that for the starting states $\underline{0}$ or $\underline{1}$, a finite number of steps are necessary and sufficient for convergence. However, if we begin the binary Burnside chain at any other starting point, the starting distribution is not uniform within cycles and thus it may take more steps before convergence to stationarity than in the orbit chain. Thus, we would like to find the eigenvectors and eigenvalues of the Markov chain on the full state space.

Our first step for diagonalizing the full chain is to write out the highest-degree Chebyshev polynomial in an explicit form:

\begin{proposition}\label{finitedifferences}
The discrete Chebyshev polynomial $T_n^n$ on $\{0, 1, \cdots, n\}$ of degree $n$ satisfies $T_n^n(i) = (-1)^i \binom{n}{i}$ for all $i \in \{0, 1, \cdots, n\}$.
\end{proposition}
\begin{proof}
The discrete Chebyshev polynomials satisfy the recurrence relation (plugging in $\alpha = \beta = 0$ into \cite[Eq. (6.2.8)]{ismail})
\[
    T_n^0(x) = 1, \quad T_n^1(x) = \frac{n-2x}{n}, \quad (j+1) (n-j) T_n^{j+1}(x) = (2j+1)(n-2x)T_n^j(x) - j(j+n+1)T_n^{j-1}(x)
\]
for all $1 \le j \le n-1$. Since $(j+1)(n-j) = (2j+1)n - j(j+n+1)$, the constant term of $T_n^j$ is $1$ for any $j$. Also, since the $T_n^m$s are orthogonal with respect to the uniform distribution on $\{0, 1, \cdots, n\}$, we have $\sum_{i = 0}^n T_n^n(i) T_n^m(i) = 0$ for all $m < n$ and therefore
\[
    \sum_{i = 0}^n T_n^n(i) f(i) = 0 \quad \text{for all polynomials }f\text{ of degree at most }(n-1).
\]
But for any such $f$, the $n$th finite difference of $f$ started at $0$ is exactly $\binom{n}{0} f(0) - \binom{n}{1} f(1) + \binom{n}{2} f(2) - \cdots$, so setting $T_n^n(i) = (-1)^i \binom{n}{i}$ satisfies the necessary orthogonality and that the constant term is $1$. Since we can perform Lagrange interpolation through these $(n+1)$ points to get a polynomial of degree at most $n$, this must be the desired polynomial.
\end{proof}

We now combine this with \cref{lumpedeigenvalues} to construct explicit expressions for our eigenvectors.

\begin{proof}[Proof of \cref{eigenvectors}]
Let $m = 2k$ be any even integer less than or equal to $n$. First, we claim that there is a (right) eigenvector of eigenvalue $\beta_k$ for the binary Burnside process on $C_2^m$ of the form
\[
    g_m(x) = T_m^m(|x|) = (-1)^{|x|} \binom{m}{|x|}
\]
(last equality by \cref{finitedifferences}). Indeed, for any $x$, we have (recall that $K_m$ denotes the transition matrix for the process on $C_2^m$)
\begin{align}\label{lumpedtounlumped}
    K_{m} g_m(x) &= \sum_{y \in C_2^{m}} g_m(y) K_{m}(x, y)\nonumber \\
    &= \sum_{i=0}^{m} T_m^{m}(i) \sum_{y: |y| = i} K_{m}(x, y)\nonumber \\ 
    &= \sum_{i=0}^{m} T_m^{m}(i) K_{m}^{\text{lumped}}(|x|, i), 
\end{align}
and by the eigenvalue equation on the lumped chain, this last expression is $\beta_k T_m^{m}(|x|) = \beta_k g_m(x)$, as desired. 

Next, choosing any size-$m$ subset $S \in \binom{[n]}{m}$ of the coordinates, we can lift $g_m$ to a function on $C_2^n$ by defining
\[
    f_S(x) = g_m(x_S) = (-1)^{|x_S|} \binom{m}{|x_S|},
\]
where $x_S$ is the restriction of $x$ to the coordinate set $S$. We claim this is a (right) eigenvector of $K_n$. Indeed, for any $x \in C_2^n$, we have
\begin{align*}
    K_nf_S(x) &= \sum_{y \in C_2^n} f_S(y) K_n(x, y) \\
    &= \sum_{y_S \in C_2^{m}} \sum_{y_{\overline{S}} \in C_2^{n - m}} (-1)^{|y_S|} \binom{m}{|y_S|} K_n(x, y) \\
    &= \sum_{y_S \in C_2^{m}} (-1)^{|y_S|} \binom{m}{|y_S|} K_{m}(x_S, y_S)
\end{align*}
since by \cref{restricttocoordinates} $K_n$ lumps to $K_{m}$ when only restricted to the coordinate set $S$. And now this last expression is exactly $\sum_{y_S} g_m (y_S) K_{m}(x_S, y_S)$, so by the eigenvalue equation it evaluates to $\beta_k g_m(x_S) = \beta_k f_S(x)$, as desired.

\medskip

This means that for each of the $\binom{n}{m}$ subsets $S$, we get an eigenvector $f_S$ for $K_n$ of eigenvalue $\beta_k$. Now we prove that $\{f_S: S \in \binom{[n]}{m}\}$ is a linearly independent set for each fixed $m$ via the following steps:

\begin{enumerate}
    \item \textit{\underline{Rephrasing the problem.}} The $\binom{n}{m}$ eigenvectors $f_S$ may be written in a $\binom{n}{m}$ by $2^n$ matrix $M$, so that the rows are indexed by subsets $S \in \binom{[n]}{m}$ and the columns are indexed by states $x \in C_2^n$. In other words, define
    \[
        M = (M_{S, x})_{S \in \binom{[n]}{m}, x \in C_2^n}, \quad M_{S, x} = (-1)^{|x_S|} \binom{m}{|x_S|}, 
    \]
    where $|x_S|$ again denotes the number of ones in $x$ among the coordinates $S$. We then wish to prove that the rows (eigenvectors) of $M$ are linearly independent, or equivalently that the matrix has full column rank. 
    \item \textit{\underline{Using symmetry.}} To prove that the matrix has full column rank, it suffices to show that some linear combination of the column vectors is $1$ in the entry for the subset $S = \{1, 2, \cdots, m\}$ and $0$ in all others. (Then by permuting the role of $1, 2, \cdots, n$ in the coefficients of those vectors, we can get all other subsets $S$ as well.) Fixing notation, let $v_x$ be the column vector of $M$ corresponding to the state $x \in C_2^n$.
    \item \textit{\underline{Constructing nice linear combinations.}} Let $d = \min\{m, n-m\}$. For all integers $0 \le a, b \le d$, consider the linear combinations of vectors
    \[
        \mathbf{v}_a = \sum_{\substack{\text{states }x \text{ with }a \text{ ones in the first } \\m \text{ coordinates and no other ones}}} v_x
    \]
    and the sets of subsets
    \[
        \mc{S}_b = \left\{S \in \binom{[n]}{m}: \Big|S \cap \{m+1, \cdots, n\}\Big|  = b\right\}.
    \]
    We have $\binom{[n]}{m} = \bigcup_{b=0}^d \mc{S}_b$. Now we claim each $\mathbf{v}_a$ is constant on each $\mc{S}_b$ and that the value it takes on any $S \in \mc{S}_b$ is $n_{ab} = \sum_{r = 0}^a (-1)^r \binom{m}{r} \binom{m-b}{r} \binom{b}{a-r}$. Indeed, among all states $x$ with $a$ ones in the first $m$ coordinates (and no other ones), there are $\binom{m-b}{r} \binom{b}{a-r}$ of them with $|x_S| = r$ (because $m-b$ of the numbers $\{1, \cdots, m\}$ are in each $S \in \mc{S}_b$), and the value of $v_x$ at $S$ for each of them is $(-1)^r \binom{m}{r}$. 

    Therefore, each of $\mathbf{v}_0, \cdots, \mathbf{v}_d$ is encoded by its $(d+1)$ values on $\mc{S}_0, \cdots, \mc{S}_d$, and we would like to show that some linear combination of $\mathbf{v}_0, \cdots, \mathbf{v}_d$ yields $1$ on $\mc{S}_0 = \{\{1, \cdots, m\}\}$ and $0$ on all other $\mc{S}_j$s. To do this, it suffices to prove that the $(d+1) \times (d+1)$ matrix 
    \[
        N = (n_{ab})_{a,b=0}^d, \quad n_{ab} = \sum_{r = 0}^a (-1)^r \binom{m}{r} \binom{m-b}{r} \binom{b}{a-r}
    \]
    is invertible.
    \item \textit{\underline{Making use of the binomial coefficients.}} We now show that row $a$ of the matrix $N$ is a degree $a$ polynomial; more precisely, there is some degree $a$ polynomial $f_a$ such that $n_{ab} = f_a(b)$ for all $b$. Indeed, for any fixed $a$, $\binom{m-b}{r}$ is a polynomial of degree $r$ in $b$, and $\binom{b}{a-r}$ is a polynomial of degree $(a-r)$ in $b$, so their product is a polynomial of degree $a$ and thus the whole expression is a polynomial. Furthermore, summing over all $r$, the total coefficient of $b^a$ in $n_{ab}$ is
    \[
        c_a = \sum_{r =0}^a (-1)^r \binom{m}{r} \frac{(-1)^r}{r!} \frac{1}{(a-r)!} = \frac{1}{a!} \sum_{r =0}^a \binom{m}{r} \binom{a}{a-r} = \frac{1}{a!} \binom{m+a}{a}.
    \]
    In particular, the coefficient of $b^a$ is nonzero and so the polynomial is indeed of degree $a$. Therefore by row reduction (subtracting off earlier rows from later ones to remove lower-order terms), the determinant of $N$ is equal to the determinant of the matrix $(c_a b^a)_{a, b = 0}^d$, which is a nonzero constant times a nonzero Vandermonde determinant, hence nonzero. This proves that we can indeed find a valid linear combination to get the desired column vector, completing the proof of linear independence.
\end{enumerate}

Thus, we have constructed a $\binom{n}{m}$-dimensional eigenspace of eigenvalue $\beta_k$ for any $0 \le k \le \frac{n}{2}$. Since eigenvectors of different eigenvalues are orthogonal, the span of these eigenspaces has dimension $\sum_{m\text{ even}} \binom{n}{m} = 2^{n-1}$. Furthermore, by construction, notice that any function $f$ in this span satisfies $f(x) = f(\overline{x})$ for all $x \in C_2^n$, where $\overline{x}$ flips each coordinate of $x$ from $0$ to $1$ or vice versa. This means that all of these eigenvectors are orthogonal to the $2^{n-1}$-dimensional space of functions 
\[
    \left\{h \in L(C_2^n): h(x) = -h(\overline{x}) \text{ for all }x \in C_2^n\right\}.
\]
Any such $h$ is an eigenvector of $K$ of eigenvalue zero, because using that $K(x, y) = K(x, \overline{y})$ for any $x, y$, we have
\[
    Kh(x) = \sum_{y \in C_2^n} h(y) K(x, y) = \sum_{y \in C_2^n} h(\overline{y}) K(x, \overline{y}) = -\sum_{y \in C_2^n} h(y) K(x, y) = -Kh(x) \implies Kh(x) = 0.
\]
So the sum of the dimensions of the constructed spaces is $2^{n-1} + 2^{n-1} = 2^n$, meaning we have constructed a basis of $C_2^n$ of eigenvectors with the stated multiplicities, completing the proof.
\end{proof}

As mentioned in the introduction, the eigenvectors $f_S$ in our proof of \cref{eigenvectors} are not orthogonal, and in fact we can describe their inner products. First of all, we have $\langle f_S, f_{S'} \rangle = 0$ for any two sets $S$ and $S'$ with $|S| \ne |S'|$ even, but this does not hold if $|S| = |S'|$. In the latter case, denote $|S| = |S'| = m$ and $|S \cap S'| = \ell$. Under the inner product $\langle f, g \rangle = \sum_x f(x) g(x) \pi(x)$ with $\pi(x) = \frac{1}{n+1} \frac{1}{\binom{n}{|x|}}$, we have
\begin{align}\label{innerproductsum}
\langle f_S, f_{S'} \rangle &=\sum_x \pi(x) f_S(x) f_S'(x) \nonumber\\ 
&= \sum_{a=0}^\ell \sum_{b=0}^{m-\ell} \sum_{c=0}^{m-\ell} \sum_{d=0}^{n-2m+\ell} \binom{\ell}{a} \binom{m-\ell}{b} \binom{m-\ell}{c} \binom{n-2m+\ell}{d} \frac{1}{n+1} \frac{1}{\binom{n}{a+b+c+d}} (-1)^{b+c} \binom{m}{a+b} \binom{m}{a+c},
\end{align}
where $a, b, c, d$ respectively represent the number of coordinates among the sets $S \cap S'$, $S \setminus S'$, $S' \setminus S$, and $(S \cup S')^c$ where $x$ has a $1$.

\begin{proposition}\label{innerproductwz}
For any even integer $m$, any $n \ge m$, and any two subsets $S, S' \subseteq \binom{[n]}{m}$ with $|S \cap S'| = \ell$, let $\overline{f_S} = \frac{f_S}{\langle f_S, f_S \rangle^{1/2}}$ and $\overline{f_{S'}} = \frac{f_S'}{\langle f_S', f_S' \rangle^{1/2}}$ be the $\ell^2(\pi)$-normalized vectors for $S$ and $S'$. Then
\[
    \langle \overline{f_S}, \overline{f_{S'}} \rangle = \frac{1}{\binom{2m+1-\ell}{m+1}}.
\]
Furthermore, for any even integer $m \le n$ and any $S \subseteq \binom{[n]}{m}$,
\[
    \langle f_S, f_S \rangle = \frac{1}{m+1} \binom{2m}{m}.
\]
\end{proposition}

Given the nice form of these inner products, it is natural to hope for closed-form expressions for linear combinations of the $f_S$s which indeed yield an orthonormal basis. As mentioned in the introduction, we defer this procedure to the companion paper \cite{diaconislinram2}.

\cref{innerproductwz} is proved using the computer-assisted ``creative telescoping'' algorithm. The key ideas of this algorithm come from Wilf and Zeilberger's \textit{WZ method} -- an overview can be found in \cite{wzmonthly} -- and subsequent work has been done to speed up the algorithm with various heuristics and a careful ansatz \cite{koutschan}. The Mathematica package \texttt{HolonomicFunctions} that we used, along with further literature references, may be found at \url{https://risc.jku.at/sw/holonomicfunctions/}.

We first show how to perform the simpler computation, the normalizing factor $\langle f_S, f_S \rangle$, without needing this machinery. (This illustrates the concept of ``showing that a certain quantity is independent of one of its parameters.'') To do this, we first note down a useful binomial coefficient computation: 

\begin{lemma}\label{binomcoeffcalculation}
For any $c_1, c_2, c_3 \in \ZZ_{\ge 0}$ with $c_1 \ge c_2$, we have the identity
\[
    \frac{1}{c_3+c_1+1}\sum_{i=0}^{c_3} \frac{\binom{c_3}{i}}{\binom{c_3+c_1}{i+c_2}} = \frac{1}{c_1+1} \frac{1}{ \binom{c_1}{c_2}} .
\]
\end{lemma}
\begin{proof}
We have
\begin{align*}
    \sum_{i=0}^{c_3} \frac{\binom{c_3}{i} \binom{c_1}{c_2}}{\binom{c_3+c_1}{i+c_2}} &= \sum_{i=0}^{c_3} \frac{c_3!c_1!}{(c_3+c_1)!} \cdot \frac{(i+c_2)!}{i! c_2!} \cdot \frac{(c_3-i+c_1-c_2)!}{(c_3-i)!(c_1-c_2)!} \\
    &= \frac{1}{\binom{c_1+c_3}{c_1}} \sum_{i=0}^{c_3} \binom{c_2+i}{c_2} \binom{c_3 + c_1 - (c_2+i)}{c_1 - c_2} \\
    &= \frac{1}{\binom{c_1+c_3}{c_1}} \cdot \binom{c_3+c_1+1}{c_1+1} \\
    &= \frac{c_3 + c_1 + 1}{c_1+1},
\end{align*}
where between the second and third lines we use the identity $\sum_{i=d-c}^{a-b} \binom{a-i}{b} \binom{c+i}{d} = \binom{a+c+1}{b+d+1}$. Now rearranging the equality between the first and last expression yields the result.
\end{proof}

\begin{proof}[Proof of the formula for $\langle f_S, f_S \rangle$ in \cref{innerproductwz}]
We are considering the case $\ell = m$ of \cref{innerproductsum}, so that our computation simplifies to
\begin{align*}
    \langle f_S, f_S \rangle &= \sum_{a=0}^m \sum_{d=0}^{n-m} \binom{m}{a} \binom{n-m}{d} \frac{1}{n+1} \frac{1}{\binom{n}{a+d}} \binom{m}{a}^2 \\
    &= \sum_{a=0}^m \binom{m}{a}^3 \cdot \frac{1}{n+1}\sum_{d=0}^{n-m}  \frac{\binom{n-m}{d}}{\binom{n}{a+d}}.
\end{align*}
Then using \cref{binomcoeffcalculation} with $c_3 = n-m$, $c_1 = m$, and $c_2 = a$ yields
\begin{align*}
    \langle f_S, f_S \rangle &= \sum_{a=0}^m \binom{m}{a}^3 \frac{1}{m+1} \cdot \frac{1}{\binom{m}{a}} \\
    &= \frac{1}{m+1}\sum_{a=0}^m \binom{m}{a}^2 \\
    &= \frac{1}{m+1} \binom{2m}{m},
\end{align*}
as desired.
\end{proof}

We now turn to the formula for $\langle \overline{f_S}, \overline{f_{S'}} \rangle$. While plugging \cref{binomcoeffcalculation} into \cref{innerproductsum} does prove that $\langle f_S, f_{S'} \rangle$ is independent of $n$ (since it allows us to perform the sum over the index $d$), computing the remaining sum by hand requires more sophisticated techniques. 

The ideas of the following argument are due to Laurent Bartholdi and refined with the help of AI tools, but the proof is written fully by the authors.

\begin{proof}[Proof of the formula for $\langle \overline{f_S}, \overline{f_{S'}} \rangle$ in \cref{innerproductwz}]
Since the quadruple sum in \cref{innerproductsum} is independent of $n$ by \cref{binomcoeffcalculation}, we select $n = 2m - \ell$ and rewrite $\ell = m-s$. Rearranging, we thus have that
\begin{align*}
    &\binom{2m+1-\ell}{m+1} \langle \overline{f_S}, \overline{f_{S'}} \rangle \\
    &= \sum_{a=0}^{m-s} \sum_{b=0}^{s} \sum_{c=0}^{s} \binom{m-s}{a} \binom{s}{b} \binom{s}{c} (-1)^{b+c} \binom{m}{a+b} \binom{m}{a+c} \frac{(a+b+c)!(m+s-a-b-c)!m!}{s!(2m)!},
\end{align*}
and we wish to prove that this quantity is identically $1$ for all $m \ge 0$ and $0 \le s\le m$.

We will do so by induction on $s$. For the base case $s = 0$, observe that we must have $b = c = 0$ and thus the triple sum evaluates to
\begin{align*}
    \sum_{a=0}^{m} \binom{m}{a} \binom{m}{a} \binom{m}{a} \frac{a! (m-a)! m!}{(2m)!}
     &= \frac{m!^2}{(2m)!}\sum_{a=0}^m \binom{m}{a}^2 \\
     &= \frac{1}{\binom{2m}{m}} \binom{2m}{m} \\
     &= 1.
\end{align*}
Now for the inductive step, let $H(m, s, a, b, c)$ be the summand above, except with factorials $x!$ replaced by gamma functions $\Gamma(x+1)$, binomial coefficients $\binom{x}{y}$ similarly replaced by products of gamma and reciprocal gamma functions $\Gamma(x+1)\frac{1}{\Gamma(y+1)} \frac{1}{\Gamma(x-y+1)}$, and $(-1)^{b + c}$ replaced with $e^{\pi i(b + c)}$. Also define the functions $G_A, G_B, G_C$ via
\[
    G_A = H \cdot\frac{a}{m-s},
\]
\[
    G_B = H \cdot \frac{b(m - s - a)(a + b + c - m - s - 1)}{(s + 1 - b) (s + 1 - c) (m - s)},
\]
\[
    G_C = H \cdot \frac{c(a+b-m)(a-m+s)}{(1 + a + b)(s + 1 - c)(m - s)}.
\]
To avoid poles coming from the denominators, let $H_\varepsilon = H(m, s, a, b + \varepsilon, c+ \varepsilon)$ for $\varepsilon \in (0, \frac{1}{2})$, and similarly define $G_{A,\varepsilon}, G_{B,\varepsilon}, G_{C,\varepsilon}$ by using $H_\varepsilon, b + \varepsilon, c + \varepsilon$ in place of $H, b, c$. Then we claim that we have the equality of meromorphic functions
\[
    H_{\varepsilon}(s+1) - H_{\varepsilon}(s) = G_{A, \varepsilon}(a+1) - G_{A,\varepsilon}(a) + G_{B,\varepsilon}(b+1) - G_{B,\varepsilon}(b) + G_{C,\varepsilon}(c+1) - G_{C,\varepsilon}(c),
\]
where in all expressions, all functions implicitly depend on the remaining parameters. (Indeed, this can be checked by dividing both sides by $H_{\varepsilon}(s)$, which leaves rational functions in the variables, and then clearing denominators so that we have equal polynomials on both sides.) Therefore, for any $0 \le s \le m - 1$ and any $\varepsilon \in (0, \frac{1}{2})$, we have (note the box of summation $[0, m-s-1] \times [0, s]^2$)
\begin{align*}
    \sum_{a=0}^{m-s-1} &\sum_{b=0}^{s} \sum_{c=0}^{s} \left(H_\varepsilon(s+1) - H_\varepsilon(s)\right) = \sum_{a=0}^{m-s-1} \sum_{b=0}^{s} \sum_{c=0}^{s} \left(G_{A,\varepsilon}(a+1) - G_{A,\varepsilon}(a)\right) \\ &+ \sum_{a=0}^{m-s-1} \sum_{b=0}^{s} \sum_{c=0}^{s}\left(G_{B,\varepsilon}(b+1) - G_{B,\varepsilon}(b)\right) + \sum_{a=0}^{m-s-1} \sum_{b=0}^{s} \sum_{c=0}^{s}\left(G_{C,\varepsilon}(c+1) - G_{C,\varepsilon}(c)\right),
\end{align*}
because for this range of summation, there are no poles coming from the gamma functions in $H_{\varepsilon}$ and no poles from any of the factors in the other denominators. Performing the telescoping sums on the right-hand side, we thus have
\begin{align}\label{telescopingHequation}
    \sum_{a=0}^{m-s-1} &\sum_{b=0}^{s} \sum_{c=0}^{s} \left(H_\varepsilon(s+1) - H_\varepsilon(s)\right) = \sum_{b=0}^{s} \sum_{c=0}^{s} \left(G_{A,\varepsilon}(m-s) - G_{A,\varepsilon}(0)\right) \nonumber \\ &+ \sum_{a=0}^{m-s-1} \sum_{c=0}^{s}\left(G_{B,\varepsilon}(s+1) - G_{B,\varepsilon}(0)\right) + \sum_{a=0}^{m-s-1} \sum_{b=0}^{s} \left(G_{C,\varepsilon}(s+1) - G_{C,\varepsilon}(0)\right).
\end{align}
Now we examine each of the right-hand terms separately:
\begin{itemize}
    \item We have $G_{A, \varepsilon}(0), G_{B, \varepsilon}(0), G_{C, \varepsilon}(0) \to 0$ as $\varepsilon \to 0$ because of the factors of $a, b, c$ in the numerators, so these terms contribute nothing to the limiting sum. 
    \item We have $G_{A, \varepsilon}(m, s, m-s, b, c) = H_{\varepsilon} \cdot \frac{m-s}{m-s} = H_{\varepsilon}(m, s, m-s, b, c)$.
    \item For $G_{B, \varepsilon}(s+1)$, we may verify that
    \[
        G_{B,\varepsilon}(m, s, a, b, c) = -\frac{b + \varepsilon}{s+1} H_{\varepsilon}(m, s+1, a, b, c).
    \]
    So plugging in $b = s + 1$ and then taking $\varepsilon \to 0$, this means $\lim_{\varepsilon \to 0} G_{B, \varepsilon}(m, s, a, s+1, c) = -H(m, s+1, a, s+1, c)$.
    \item Finally, for $G_{C, \varepsilon}(s+1)$, we may verify that
    \[
        G_{C,\varepsilon}(m, s, a, b, c)  = -\frac{(s+1-b-\varepsilon)(a+b+\varepsilon-m)(c+\varepsilon)}{(s+1)(m+s+1-a-b-c-2\varepsilon)(1+a+b+\varepsilon)} H_{\varepsilon}(m, s+1, a, b, c).
    \]
    So plugging in $c = s + 1$ and then taking $\varepsilon \to 0$, this means $\lim_{ \varepsilon \to 0} G_{C, \varepsilon}(m, s, a, b, s+1) = \frac{s+1-b}{1+a+b} H(m, s+1, a, b, s+1)$.
\end{itemize}
Substituting back into \cref{telescopingHequation} and then taking $\varepsilon \to 0$ thus yields an equation only involving $H$:
\begin{align*}
\sum_{a=0}^{m-s-1} &\sum_{b=0}^{s} \sum_{c=0}^{s} H(m, s+1, a, b, c) - \sum_{a=0}^{m-s-1} \sum_{b=0}^{s} \sum_{c=0}^{s} H(m, s, a, b, c) = \sum_{b=0}^s \sum_{c=0}^s H(m, s, m-s, b, c) \\
&- \sum_{a=0}^{m-s-1} \sum_{c=0}^s H(m, s+1, a, s+1, c) + \sum_{a=0}^{m-s-1} \sum_{b=0}^s \frac{s+1-b}{1+a+b} H(m, s+1, a, b, s+1).
\end{align*}
Observe that if we move all terms with second argument $s+1$ to the left and all terms with second argument $s$ to the right, we will have almost shown the desired inductive step
\[
    \sum_{a=0}^{m-s-1} \sum_{b=0}^{s+1} \sum_{c=0}^{s+1} H(m, s+1, a, b, c) = \sum_{a=0}^{m-s} \sum_{b=0}^s \sum_{c=0}^s H(m, s, a, b, c).
\]
Indeed, it now remains to show that
\[
    - \sum_{b=0}^s \frac{s+1-b}{1+a+b} H(m, s+1, a, b, s+1) = \sum_{b=0}^{s+1} H(m, s+1, a, b, s+1)
\]
and then sum this over $a \in [0, m-s-1]$ to fill in the ``missing slice'' $[0, m-s-1] \times [0, s+1] \times \{s+1\}$. Moving all terms to one side, this last identity is equivalent to 
\[
    \sum_{b=0}^{s+1} \frac{2+a+s}{1+a+b} H(m, s+1, a, b, s+1) = 0,
\]  
which (after canceling out common factors that don't depend on $b$ from all terms) is itself equivalent to 
\[
    \sum_{b=0}^{s+1} (-1)^b \binom{s+1}{b} (a+b+2)_s = 0.
\]
But this is true because $(a+b+2)_s$ is a degree-$s$ polynomial in $b$, and the left-hand side is the $(s+1)$th finite difference of it, hence zero. Thus the inductive step shows that our summation is indeed always $1$ for all $0 \le s \le m$, completing the proof.
\end{proof}

\begin{remark}
The Gram matrix of \cref{innerproductwz} is so neat that it is natural to try to understand it directly. Richard Stanley points out that matrices indexed by the size-$m$ subsets of $[n]$ with $(S, S')$ entry given by $g(|S \cap S'|)$ (for any function $g$) are elements of the adjacency algebra of the Johnson scheme. Such matrices can be explicitly diagonalized with dual Hahn polynomial eigenfunctions. A clear, elementary account can be found in \cite{johnsonscheme}; in particular, Theorem 1.1 in that paper shows that the eigenvalues of our Gram matrix are the rational numbers
\[
    \lambda_t = \sum_{\ell = 0}^m g(\ell) \sum_{i=0}^t (-1)^{t-i} \binom{m-i}{\ell-i} \binom{n-m+i-t}{m-\ell+i-t} \binom{t}{i}, \quad 0 \le t \le m,
\]
where $g(\ell) = \frac{1}{\binom{2m+1-\ell}{m+1}}$ are our inner products, and where the $\lambda_t$-eigenspace has dimension $\binom{n}{t} - \binom{n}{t-1}$.  

Thus from here, a different approach to orthonormalization is possible: if $\{v_i\}$ are independent column vectors with Gram matrix $G = (\langle v_i, v_j \rangle)_{ij}$, then an orthonormal basis is given by $u_i = \sum_j (G_{ji})^{-1/2} v_j$. (In matrix notation, we have $U = VG^{-1/2}$ where $U, V$ have columns $\{u_i\}$ and $\{v_i\}$.) However, the presence of $G^{-1/2}$ suggests a messy calculation and thus we have not tried to push this further.
\end{remark}

\section{Proof of \cref{chisquare}}\label{chisquareboundsection}

In this section, we compute the average chi-square distance and show that order $\frac{n}{\log n}$ steps are required to be close to stationarity. The argument uses properties of $\chi_{\text{avg}}^2(\ell)$, which we recall here.

As mentioned in the introduction, the average chi-square distance to stationarity after $\ell$ steps is given by
\[
    \chi_{\text{avg}}^2(\ell) = \sum_{x \in \mf{X}} \pi(x) \chi_x^2(\ell) =  \sum_{x, y \in \mf{X}} \left|\frac{K^\ell(x, y)}{\pi(y)} - 1\right|^2 \pi(x) \pi(y).
\]
We made use of eigenvalues and eigenvectors in \cref{analyticsubsection} to rewrite chi-square distance in terms of eigenvectors. Again letting $f_0, \cdots, f_{|\mf{X}| - 1}$ be an orthonormal basis of eigenvectors for $\ell^2(\pi)$, with $f_0$ corresponding to the trivial eigenvalue of $1$, we have
\begin{align}\label{chisquarebound}
    \chi_{\text{avg}}^2(\ell) &= \sum_{x \in \mf{X}} \pi(x) \sum_{i = 1}^{|\mf{X}| - 1} f_i(x)^2 \beta_i^{2\ell} \nonumber \\
    &= \sum_{i = 1}^{|\mf{X}| - 1} \beta_i^{2\ell} \sum_{x \in \mf{X}} \pi(x) f_i(x)^2 \nonumber  \\
    &= \sum_{i = 1}^{|\mf{X}| - 1} \beta_i^{2\ell}.
\end{align}
Thus, we may plug in our known eigenvalues and multiplicities directly and obtain estimates for various $\ell$. We will use this to first prove the results for average chi-square distance:

\begin{proof}[Proof of \cref{chisquare}, parts (1) and (2)]
For (1), we use the eigenvalue multiplicities from \cref{eigenvectors}, along with \cref{chisquarebound}, to get
\[
    \chi_{\text{avg}}^2(\ell) = \sum_{k = 1}^{\lfloor n/2 \rfloor} \binom{n}{2k} \beta_k^{2\ell},
\]
where $\beta_k = \frac{1}{2^{4k}} \binom{2k}{k}^2$ as in \cref{lumpedeigenvalues}. We will lower and upper bound this quantity to show that order $\frac{n}{\log n}$ steps are necessary and sufficient -- in fact, we will show a sharp cutoff at $\frac{\log 2}{2} \frac{n}{\log n}$.

\medskip

For (2), because of the bounds on the central binomial coefficient (\cite[Theorem 2.5]{stanica})
\begin{equation}\label{stanicabound}
    \frac{4^a}{\sqrt{\pi a}} \exp\left(-\frac{1}{8a}\right) <  \binom{2a}{a} < \frac{4^a}{\sqrt{\pi a}},
\end{equation}
we have the bounds on our eigenvalues
\[
    \frac{1}{\pi k} \exp\left(-\frac{1}{4k}\right) < \beta_k < \frac{1}{\pi k}.
\]
In particular, since we will only use the lower bound for $k$ of order $n$, $\exp\left(-\frac{1}{4k}\right)^{2\ell}$ is of constant order for the values of $\ell \sim \frac{n}{\log n}$ and $k$ that we will need. Thus it suffices to prove (2) with $\chi_{\text{avg}}^2(\ell)$ replaced by 
\[
    \overline{\chi}_{\text{avg}}^2(\ell) = \sum_{k=1}^{\lfloor n/2 \rfloor} \binom{n}{2k} \frac{1}{(\pi k)^{2\ell}}.
\]
For simplicity of notation, take $n$ to be a multiple of $4$; the general case can be handled analogously by offsetting some indices (for instance, using one of the middle two binomial coefficients instead of the central binomial coefficient). 

First, we show that for $\ell = (1 - \varepsilon) \frac{\log 2}{2} \frac{n}{\log n}$, we have $\overline{\chi}_{\text{avg}}^2(\ell) \to \infty$. Indeed, the middle term ($k = \frac{n}{4}$) of the summation has asymptotics
\begin{align*}
    \binom{n}{n/2} \frac{1}{(\frac{\pi}{4} n)^{2\ell}} &\sim \frac{2^n}{\sqrt{\pi n/2}} \exp\left(-(1 - \varepsilon) \log 2 \frac{n}{\log n} \log \left(\frac{\pi}{4} n\right)\right) \\
    &= \frac{2^{\varepsilon n} 2^{(\varepsilon - 1) \log\left(\frac{\pi}{4}\right) \frac{n}{\log n}}}{\sqrt{\pi n/2}},
\end{align*}
and this expression indeed diverges to $+\infty$ as $n \to \infty$. 

\medskip

For the upper bound, we will choose $2\ell = \left(1 + \frac{c}{\log n}\right) \frac{n \log 2}{\log (\frac{\pi}{4}n)}$, where $c = c_n$ is of order $1$ and is determined in \cref{l2cutoffestimate2}. (The result then follows if we show that the sum tends to zero for this choice of $\ell$.) If we parametrize $k = \frac{n}{4} + j$, meaning that $2k = \frac{n}{2} + 2j$ (and $-\frac{n}{4} < j < \frac{n}{4}$), then we have
\begin{align}\label{l2cutoffestimate1}
\frac{1}{(\pi k)^{2\ell}} &= e^{-2\ell \log(\pi k)} \nonumber \\
&= \exp\left(-\left(1 + \frac{c}{\log n}\right) \frac{n \log 2}{\log \left(\frac{\pi}{4}n\right)} \log \left(\pi\left(\frac{n}{4} + j\right)\right) \right) \nonumber \\
&= \exp\left(-\left(1 + \frac{c}{\log n}\right) \frac{n \log 2}{\log \left(\frac{\pi}{4}n\right)} \left(\log \left(\frac{\pi}{4} n\right) + \log\left(1 + \frac{4j}{n}\right)\right)\right) \nonumber \\
&= \exp\left(-\left[n \log 2 + \frac{n \log 2}{\log\left(\frac{\pi}{4}n\right)}\log\left(1 + \frac{4j}{n}\right) + \frac{c}{\log n} n \log 2 + \frac{c}{\log n}  \frac{n \log 2}{\log\left(\frac{\pi}{4}n\right)}\log\left(1 + \frac{4j}{n}\right)\right]\right).
\end{align}

This last expression in \cref{l2cutoffestimate1} must be multiplied by $\binom{n}{\frac{n}{2} + 2j}$ and then summed over $-\frac{n}{4} < j < \frac{n}{4}$, and we must show that the sum tends to zero. This is shown in zones:

\begin{itemize}
\item \textit{\underline{Zone 1: $0 \le j < \frac{n}{4}$}}. For these cases, we can bound the binomial coefficient crudely by $2^n$ (which cancels out the first term in the exponent of \cref{l2cutoffestimate1}). Using that $\frac{x}{1+x} < \log(1+x) < x$ for all $x > 0$ and plugging in $x = \frac{4j}{n}$ yields that $-\log\left(1 + \frac{4j}{n}\right) < -\frac{4j/n}{1 + 4j/n} < -\frac{2j}{n}$. Thus, the second and fourth terms in the exponent of \cref{l2cutoffestimate1} may be bounded via
\[
    -\frac{n \log 2}{\log\left(\frac{\pi}{4}n\right)}\log\left(1 + \frac{4j}{n}\right) < -\frac{(2 \log 2)j}{\log\left(\frac{\pi}{4} n\right)} \qquad \text{and }
\]
\[
    -\frac{c}{\log n}  \frac{n \log 2}{\log\left(\frac{\pi}{4}n\right)}\log\left(1 + \frac{4j}{n}\right) < -\frac{c}{\log n}  \frac{(2 \log 2)j}{\log\left(\frac{\pi}{4}n\right)}.
\]
Further observe that for any $A > 0$, we have
\[
    \sum_{j=0}^{\infty} e^{-Aj/\log n} = \frac{1}{1 -  e^{-A/\log n}} \sim \frac{\log n}{A},
\]
so that combining bounds together, the total contribution to $\overline{\chi}_{\text{avg}}^2(\ell)$ over $0 \le j < \frac{n}{4}$ is asymptotically bounded from above by
\[
    \frac{\log n}{2 \log 2} \exp\left(-\frac{cn \log 2}{\log n}\right),
\]
and (with the choice of $c$ below in \cref{l2cutoffestimate2}) this indeed goes to zero as $n \to \infty$.
\item \textit{\underline{Zone 2: $-\frac{n}{6} < j < 0$}}. Write $j' = -j$ for clarity. Because $-\log(1-y) < \frac{y}{1-y}$ for $0 < y < 1$, we have
\[
    -\log\left(1 - \frac{4j'}{n}\right) < \frac{4j'/n}{1 - 4j'/n} < \frac{12j'}{n}
\]
for all $j' < \frac{n}{6}$. If we write $\frac{n \log 2}{\log\left(\frac{\pi}{4} n\right)} \frac{12j'}{n} = \alpha j'$ (so $\alpha = \frac{12 \log 2}{\log\left(\frac{\pi}{4} n\right)})$, then
\[
    \sum_{j' = 0}^{\frac{n}{6} - 1} e^{\alpha j'} = \frac{e^{\alpha n/6} - 1}{e^{\alpha} - 1} \sim \frac{1}{\alpha}  e^{\alpha n/6} = \frac{\log\left(\frac{\pi}{4} n\right)}{12 \log 2} \exp\left(\frac{2n \log 2}{\log\left(\frac{\pi}{4} n\right)} \right).
\]
In this zone, we still bound all binomial coefficients crudely by $\binom{n}{2k} < 2^n$. Since the expression \cref{l2cutoffestimate1} has third term $-\frac{cn \log 2}{\log n}$ in the exponent, it follows that if
\begin{equation}\label{l2cutoffestimate2}
c = 2 + \frac{\log n}{n \log 2} \left(\log \log \left(\frac{\pi}{4} n\right) + \theta\right),
\end{equation}
then the total contribution to $\overline{\chi}_{\text{avg}}^2(\ell)$ over $-\frac{n}{6} < j < 0$ is bounded above by a constant times $e^{-\theta}$, which tends to zero by choosing $\theta = \theta_n$ increasing (but increasing slowly enough so that $c$ stays bounded).

\item \textit{\underline{Zone 3: $-\frac{n}{4} + n^{0.9} < j \le -\frac{n}{6}$}}. For this zone, it is important to bound the binomial coefficient. We have 
\begin{align*}
    \binom{n}{a} = \frac{n(n-1) \cdots (n-a+1)}{a!} &= \frac{n^a}{a!}\left(1 - \frac{1}{n}\right) \cdots \left(1 - \frac{a-1}{n}\right) \\
    &\le n^a \exp\left(-\frac{1}{n} \sum_{j=1}^{a-1} j\right) \\
    &\le \frac{n^a e^{-\binom{a}{2}/n}}{a!}.
\end{align*}
Thus for $a = \theta n$ for $0 < \theta < \frac{1}{2}$, we have by the above bound and Stirling's formula that
\begin{align*}
    \binom{n}{a} &\le \exp\left(a \log n - \frac{1}{2n} a^2 - a \log a + a - \frac{1}{2} \log a +  O(1)\right) \\
    &= \exp\left(-\theta n \log \theta - \frac{\theta^2}{2}n + \theta n - \frac{1}{2} \log(\theta n) + O(1)\right) \\
    &\le \exp\left(n\left(\theta - \frac{\theta^2}{2} - \theta \log \theta\right) + O(1)\right).
\end{align*}
In our expression for $\overline{\chi}_{\text{avg}}^2(\ell)$, this factor is being multiplied by $\frac{1}{\left(\frac{\pi a}{2}\right)^{2\ell}} = \exp\left(-2\ell \log \left(\frac{\pi}{2} \theta n\right)\right)$, so that if $2\ell = (1 + \varepsilon) \frac{n \log 2}{\log\left(\frac{\pi}{4} n\right)}$, then this factor is of the form 
\[
    \exp\left(-(1 + \varepsilon) \frac{n \log 2}{\log\left(\frac{\pi}{4} n\right)}\log \left(\frac{\pi}{2} \theta n\right)\right) = \exp\left(-(1 + \varepsilon) n \log 2  \left(1 + \frac{\log (2\theta)}{\log\left(\frac{\pi}{4} n\right)}\right)\right).
\]
Multiplying these exponentials together, the contribution to $\overline{\chi}_{\text{avg}}^2(\ell)$ has lead term in the exponent $n\left(\theta - \frac{\theta^2}{2} - \theta \log \theta - (1+\varepsilon) \log 2\right)$, and for $\theta < \frac{1}{6}$ this is bounded from above by $-0.2n$. Furthermore, the remaining positive factor in the exponential $\exp\left(-(1 + \varepsilon) n \log 2 \frac{\log (2\theta)}{\log\left(\frac{\pi}{4} n\right)}\right)$ is asymptotically bounded from above by $e^{0.15n}$ as long as $\theta \gtrsim n^{-0.1}$ and $\varepsilon < 1$. Thus in this zone each term is exponentially small and so the sum is also exponentially small.
\item \textit{\underline{Zone 4: $-\frac{n}{4} < j \le -\frac{n}{4} + n^{0.9}$}} (that is, $1 \le 2k \le 2n^{0.9}$). Here we can just bound $\frac{1}{\pi k} \le \frac{1}{\pi}$, so that in this zone we have
\begin{align*}
    \binom{n}{2k} \frac{1}{(\pi k)^{2\ell}} &\le n^{2k} \frac{1}{\pi^{2\ell}} \\
    &= \exp\left(2k \log n - 2\ell \log \pi\right) \\
    &\le \exp\left(2n^{0.9} \log n - \frac{n \log 2}{\log\left(\frac{\pi}{4} n\right)} \log \pi\right).
\end{align*}
Since $n^{0.9}$ times this quantity goes to zero as $n \to \infty$, the contribution to $\overline{\chi}_{\text{avg}}^2(\ell)$ in this zone also goes to zero.
\end{itemize}
Combining the bounds across the different zones yields the needed upper bound on $\overline{\chi}_{\text{avg}}^2(\ell)$ and hence $\chi_{\text{avg}}^2(\ell)$.

\medskip

Finally, we wish to upgrade the bounds above to per-state statements $\chi_x^2(\ell)$. The existence of $x^{(n)}$ such that $\chi_{x^{(n)}}(\ell)$ diverges after $(1 - \varepsilon) \frac{\log 2}{2} \frac{n}{\log n}$ steps is clear, because if $\chi_x^2(\ell)$ were bounded for all individual states, then the average could not diverge. For the other claim, observe that because $K$ and $\pi$ are invariant under the action of $S_n$, $\chi_x^2(\ell)$ is constant on each orbit $\mc{O}_i$, and $\sum_{x \in \mc{O}_i} \pi(x) = \frac{1}{n+1}$ for each orbit. Therefore,
\[
    \chi_{\text{avg}}^2(\ell) = \sum_x \pi(x) \chi_x^2(\ell) = \frac{1}{n+1} \sum_{i=0}^n \chi_{\mc{O}_i}^2(\ell),
\]
where $\chi_{\mc{O}_i}^2(\ell)$ is the chi-square distance after $\ell$ steps when started from any state in the orbit $\mc{O}_i$. So to prove that $\chi_{\mc{O}_i}^2(\ell) \to 0$ for all $i$, it suffices to prove that $(n+1)\chi_{\text{avg}}^2(\ell) \to 0$. But this is actually already implied by the calculation above, which in fact shows faster-than-polynomial decay of $\chi_{\text{avg}}^2(\ell)$ in $n$ in all zones when $\ell = (1 + \varepsilon) \frac{\log 2}{2} \frac{n}{\log n}$. This concludes the proof.
\end{proof}

\begin{remark}
The delicate choice of $c$ for zone $2$ was developed as part of an argument to prove a limiting ``shape theorem'' for $\chi_{\text{avg}}^2(\ell)$. Perhaps the calculations for different zones may be merged by directly using $\ell = (1 + \varepsilon) \frac{\log 2}{2} \frac{n}{\log n}$; however, we leave the argument in its current form in case a reader wants to work more and prove a shape theorem.
\end{remark}

It remains now to show parts (3) and (4) of the theorem; that is, most states indeed take order $\frac{n}{\log n}$ steps to reach stationarity in $\ell^2$. For additional context, recall from \cref{l2byl1example} that states with all but a constant number of $0$s or $1$s have $\ell^2$ mixing in just $\log n$ steps, and so it is natural to ask whether we can also exhibit explicit states with slow $\ell^2$ mixing. The following argument, using an idea of Lucas Teyssier, shows that states with a positive fraction of both $0$s and $1$s do in fact require $\frac{n}{\log n}$ steps to converge.

\begin{proof}[Proof of \cref{chisquare}, part (3)]
We may bound the chi-square distance starting at $x^{(n)}$ using just the term for $x^{(n)}$ itself:
\begin{align*}
    \chi_{x^{(n)}}^2(\ell) = \sum_{y \in C_2^n} \left|\frac{K^\ell(x^{(n)}, y)}{\pi(y)} - 1\right|^2 \pi(y) &\ge \left|\frac{K^\ell(x^{(n)}, x^{(n)})}{\pi(x^{(n)})} - 1\right|^2\pi(x^{(n)}) \\
    &\ge \left|\frac{K\left(x^{(n)}, x^{(n)}\right)^\ell}{\pi(x^{(n)})} - 1\right|^2 \pi(x^{(n)}).
\end{align*}
Plugging into the formula for the transition probability in \cref{closedformk}, we have that for any state $x^{(n)}$ with $|x^{(n)}| = a$ zeros and $n-a$ ones,
\[
    K(x^{(n)}, x^{(n)}) = \frac{\binom{2a}{a}\binom{2(n-a)}{n-a}}{4^n \binom{a}{a} \binom{n-a}{n-a}} \ge \exp\left(-\frac{1}{8a} - \frac{1}{8(n-a)}\right) \frac{1}{\sqrt{\pi a}} \cdot \frac{1}{\sqrt{\pi (n-a)}},
\]
using the central binomial coefficient bound in \cref{stanicabound}. Since $1 \le a \le n-1$ by our assumption on $x^{(n)}$, we thus have
\[
    K(x^{(n)}, x^{(n)}) \ge \exp\left(-\frac{1}{8} - \frac{1}{8}\right) \frac{1}{\sqrt{\pi n/2}} \cdot \frac{1}{\sqrt{\pi n/2}} > \frac{1}{en}.
\]
But because $cn \le |x^{(n)}| \le (1-c)n$ by assumption, $\frac{1}{\pi(x^{(n)})} = (n+1) \binom{n}{|x^{(n)}|}$ grows exponentially in $n$ (with constant depending on $c$), so it takes many steps for $K\left(x^{(n)}, x^{(n)}\right)^\ell$ to get small enough to be of comparable order to $\pi(x^{(n)})$. More precisely, if $\ell = \Theta_c\left(\frac{n}{\log n}\right)$, then $K(x^{(n)}, x^{(n)})^\ell > \pi(x^{(n)})^{1/3}$ and $\chi_{x^{(n)}}^2(\ell)$ is exponentially growing in $n$, as desired.
\end{proof}

We may keep more careful track of the constants in the argument above to get sharp bounds for the half-zeros, half-ones state, yielding the final part of the theorem:

\begin{proof}[Proof of \cref{chisquare}, part (4)]
Consider the specific states $y^{(n)} \in C_2^n$ with $\lfloor \frac{n}{2} \rfloor$ zeros followed by $\lceil \frac{n}{2} \rceil$ ones. Following the calculations in the proof of part (3), 
\[
    \chi_{y^{(n)}}^2(\ell) \ge \left|\frac{K\left(y^{(n)}, y^{(n)}\right)^\ell}{\pi(y^{(n)})} - 1\right|^2 \pi(y^{(n)}),
\]
now with the explicit bounds 
\[
    K(y^{(n)}, y^{(n)}) \ge e^{-1/4} \frac{1}{\sqrt{\pi \lfloor n/2 \rfloor}} \frac{1}{\sqrt{\pi \lceil n/2 \rceil}} > \frac{1}{en}\quad \text{and}
\]
\[
    \pi(y^{(n)}) = \frac{1}{(n+1) \binom{n}{\lfloor n/2 \rfloor}} < \frac{1}{2^n}.
\]
So for any $\varepsilon > 0$, if $\ell \le \left(\frac{\log 2}{2} - \varepsilon\right) \frac{n}{\log n}$, then $K\left(y^{(n)}, y^{(n)}\right)^\ell \ge 2^{-n/2} e^{\varepsilon n} 2^{-n/(2 \log n)}$. Plugging in these bounds (and also using that $\pi(y^{(n)}) \ge \frac{1}{(n+1) 2^n}$ to lower bound the last factor), we have for these $\ell$s that
\[
    \chi_{y^{(n)}}^2(\ell) \ge \left|2^n \cdot 2^{-n/2} e^{\varepsilon n} 2^{-n/(2 \log n)} - 1\right|^2 \cdot \frac{1}{(n+1) 2^n}.
\]
This right-hand side goes to infinity as $n \to \infty$, proving the desired claim.
\end{proof}

\section{Related chains}\label{relatedchainssection}

\subsection{Other Markov chains with similar properties}\label{othermarkovchains}

The binary Burnside process in this paper is a Markov chain on the hypercube which lumps to the orbits $\mc{O}_i$. Here, we mention some other chains that share this property and highlight some differences in their rates of convergence.

First of all, consider the \textit{nearest-neighbor random walk on the hypercube}. The lumped chain in that setting is the \textit{Ehrenfest urn}, in which there are two urns and a uniform ball is moved from one urn to the other at each step. While both chains can be lumped to the same orbits $\mc{O}_i$ and also both satisfy $K(x, y) = K(\sigma(x), \sigma(y))$, this lumped chain behaves quite differently from the lumped binary Burnside process. Specifically (see \cite{cutoff} for more details and references to proofs, as well as \cite{kharezhou} for some generalizations), $\frac{1}{4} n \log n + cn$ steps are necessary and sufficient for convergence in both $\ell^1$ and $\ell^2$ when started from $0$, while only $cn$ steps are necessary and sufficient when started from $\frac{n}{2}$. In contrast, explicit computations using the discrete Chebyshev polynomials show that for the lumped binary Burnside chain, a constant number of steps are necessary and sufficient both when started from $0$ and from $\frac{n}{2}$.

Continuing this comparison, we may also compare behavior of the two unlumped chains on $C_2^n$. Started from any vertex, the nearest-neighbor random walk with holding converges to stationarity in $\frac{1}{4}n \log n + cn$ steps (the exact profile is computed in \cite{nearestneighbor}), and the choice of starting state does not matter since we have a random walk on a group. In contrast, our main results show that the starting state drastically affects rates of convergence for the binary Burnside chain: \cref{chisquare} shows that most states take order $\frac{n}{\log n}$ steps to converge, while \cref{diaconiszhongbound} shows that the all-zeros state takes just a constant number of steps.

\medskip

For a second example, consider the \textit{uniform-prior beta-binomial chain}, first studied in \cite{twogibbs} as an example of a two-component Gibbs sampler. Briefly, this chain may be described as follows. Consider $(j, \theta)$ sampled from the joint distribution $f(j, \theta) d\theta = \binom{n}{j} \theta^j (1 - \theta)^{n-j} d\theta$, where $j \in \{0, 1, \cdots, n\}$ and $d\theta$ is Lebesgue measure on $[0, 1]$. (This is indeed a probability measure, since summing over $j$ yields $1$ for all $\theta$ and then integrating over $\theta$ yields $1$ overall.) We may form a Markov chain on the $j$-state space $\{0, 1, \cdots, n\}$ as follows:
\begin{itemize}
    \item From $j$, sample $\theta$ from the distribution conditioned on $j$ (which is the beta distribution with parameters $(j+1, n-j+1)$).
    \item From $\theta$, sample $j'$ from the distribution conditioned on $\theta$ (which is binomial with parameters $(n, \theta)$).
\end{itemize}
Much like one step of the binary Burnside chain consists of performing the steps $x \mapsto s \mapsto y$, one step of this chain consists of performing the steps $j \mapsto \theta \mapsto j'$. The resulting chain has a uniform stationary distribution, and \cite[Proposition 1.1]{twogibbs} shows that it also has the discrete Chebyshev polynomials as eigenvectors (just like our lumped chain, as shown in \cref{lumpedeigenvalues}). However, the eigenvalues in the beta-binomial chain do have an explicit dependence on $n$ (unlike in our problem), and the chain requires order $n$ steps to converge in chi-square distance when started from either $0$ or $n$.

We may unlump this chain to get another Markov chain on $C_2^n$ in a straightforward way:
\begin{itemize}
    \item From $x \in C_2^n$, sample $\theta$ from the beta distribution with parameters $(|x|+1, n-|x|+1)$.
    \item From $\theta$, sample $x'$ as a sequence of $n$ Bernoulli($\theta$) random variables, viewed as a binary $n$-tuple.
\end{itemize}

This chain thus also has the same constant-on-orbits stationary distribution $\pi(x) = \frac{1}{(n+1) \binom{n}{|x|}}$ as our binary Burnside process. But in this case, ``unlumping'' the chain does not lead to higher eigenvalue multiplicities or longer mixing times. Indeed, since $\theta$ depends only on the orbit of $x$, we have (letting $\tilde{K}^{\text{lumped}}$ and $\tilde{K}^{\text{unlumped}}$ denote the transition matrices of the lumped and unlumped chains, respectively)
\[
    \tilde{K}^{\text{unlumped}}_n(x, x') = \frac{1}{\binom{n}{|x'|}}\tilde{K}^{\text{lumped}}_n(|x|, |x'|),
\]
which implies that the nonzero eigenvalues and multiplicities of the unlumped chain are identical to those of the lumped chain -- all additional eigenvalues are zero. Also, the symmetry of the binary Burnside chain described in \cref{restricttocoordinates} does not hold for the unlumped beta-binomial chain. In summary, despite the identical stationary distributions and eigenvectors coming from orthogonal polynomials, these two Markov chains behave quite differently.

\subsection{Generalizing beyond the binary case}\label{beyondbinarysubsection}

This paper discusses the binary Burnside process, which is a Markov chain on the hypercube $C_2^n$. An analogous definition can also be made for a Burnside process on $(C_k^n, S_n)$ for $k \ge 2$, and we discuss how some symmetries of the binary case still persist and propose some ideas for extending our results.

In one step of this more general Burnside process, we begin with an $n$-tuple $x \in C_k^n$, uniformly pick a permutation permuting the coordinates within each value, write it as a product of disjoint cycles, and label each cycle uniformly with one of the $k$ values in the alphabet. \cref{restricttocoordinates} generalizes directly in this setting, with the only modification to the proof being that we partition the coordinate set into $k$ sets of locations rather than just the locations of zeros and ones:

\begin{proposition}\label{restricttocoordinates2}
The restriction of the Burnside process on $(C_k^n, S_n)$ to any $m \le n$ of its coordinates is also a Markov chain, and its transition probabilities are exactly given by the Burnside process on $(C_k^m, S_m)$.
\end{proposition}

Note however that the methods used to prove \cref{eigenvectors} in the $k = 2$ case from \cref{restricttocoordinates} run into significant challenges for $k > 2$. Instead of considering a lumped chain on $\{0, 1, \cdots, n\}$, we must now consider the Bose-Einstein orbit chain of \cite{boseeinstein} mentioned in \cref{propertiessection}. This Markov chain generally has irrational eigenvalues (even for small values like $k = 3, n = 6$), making explicit descriptions of the eigenvectors (as we had with the discrete Chebyshev polynomials) difficult. Additionally, the unlumped Burnside chain now exhibits eigenvalues not present in the lumped chain, meaning we cannot just ``lift'' lumped eigenvectors to unlumped ones and expect a full eigendecomposition. 

\begin{conjecture}\label{ckmultiplicityconjecture}
Fix $k$, and let $\lambda$ be any nonzero eigenvalue of the Burnside chain on $(C_k^n, S_n)$ for any $n$. Then $\lambda$ occurs with multiplicity $a_\lambda \binom{n}{b_\lambda}$ for some integers $a_\lambda, b_\lambda$.
\end{conjecture}

For the $k = 2$ case, \cref{eigenvectors} shows that for each eigenvalue $\lambda = \beta_m$, we have $a_\lambda = 1$ and $b_\lambda = 2m$, and no other eigenvalues besides the ones from the orbit chain appear. In contrast, consider the eigenvalue $\frac{1}{18}$ for $k = 3$. This eigenvalue does not appear in the orbit Bose-Einstein chain for any $n \le 8$, but it occurs with multiplicity $2, 10, 30, 70$ for $n = 4, 5, 6, 7$, suggesting that $a_\lambda = 2$ and $b_\lambda = 4$.

Even without eigenvalues and eigenvectors, some information about mixing time can still be proved. Aldous showed (as a generalization of \cref{aldousbound}) that 
\[
    ||K_x^\ell - \pi||_{\mathrm{TV}} \le n\left(1 - \frac{1}{k}\right)^\ell,
\]
meaning that $k(\log n + c)$ steps are sufficient for $\ell^1$ mixing from any starting state. Our argument for $\ell^2$ mixing also generalizes to prove an analogous bound:

\begin{proposition}\label{slowmixingwithinorbits}
Let $x^{(n)} \in C_k^n$ be any sequence of states such that $x^{(n)}$ contains at least $cn$ coordinates of two different values in $C_k$ for some $c \in (0, 1)$. Then for the Burnside process on $(C_k^n, S_n)$, we have $\chi_{x^{(n)}}^2(\ell) \to \infty$ for $\ell = \Theta_{c,k}(\frac{n}{\log n})$ (meaning that the constant may now depend on both $c$ and $k$).
\end{proposition}
\begin{proof}
As in the proof of \cref{chisquare}, part (3) (near the end of \cref{chisquareboundsection}), we bound $\chi_{x^{(n)}}^2(\ell)$ using only the term $y = x^{(n)}$. For brevity, we will not establish an exact formula for $K(x^{(n)}, x^{(n)})$ in this more general case; instead, it suffices to show the lower bound $K(x^{(n)}, x^{(n)}) \ge \frac{1}{(nk)^k}$. Indeed, the probability that the permutation fixing $x^{(n)}$ is a union of full cycles within each value in $C_k$ is at least $\frac{1}{n^k}$ (because the probability that a uniformly random permutation in $S_m$ is a single cycle is $\frac{1}{m}$), and then the probability that each of those cycles are labeled with the value in $C_k$ that they came from is at least $\frac{1}{k^k}$. 

On the other hand, $\frac{1}{\pi(x^{(n)})}$ still grows exponentially in $n$, with constant depending on $c$ and $k$. Thus $\ell = \Theta_{c,k}\left(\frac{n}{\log n}\right)$ steps are required until $K\left(x^{(n)}, x^{(n)}\right)^\ell < \pi(x^{(n)})^{1/3}$, and so $\chi_{x^{(n)}}^2(\ell)$ is still exponentially growing as $n \to \infty$ for this value of $\ell$.
\end{proof}

So even in this more general case, $\frac{n}{\log n}$ steps are necessary for $\ell^2$ mixing from most starting states. However, without eigenvalues and eigenvectors, neither the eigenvalue bound of \cref{chisquare} nor the ``$\ell^2$ by $\ell^1$ upper bound'' of \cref{l2byl1corollary} is admissible for proving that this is also sufficient. Thus, it would be interesting to prove \cref{ckmultiplicityconjecture} and provide matching upper bounds for mixing time, and the aforementioned recent progress in \cite{nestoridisly} and \cite{fengpaguyo} may provide some starting directions for careful analysis. Along those lines, we conclude with a final unified conjecture for the Burnside process:

\begin{conjecture}
For any fixed $k \ge 2$, let $K_n$ denote the Burnside process on $(C_k^n, S_n)$. Then $K_n$ has cutoff in both $\ell^1$ and $\ell^2$ when started from states with a positive limiting proportion of at least two different values in $C_k$ (as in \cref{slowmixingwithinorbits}).
\end{conjecture}

\bibliographystyle{bibstyle}
\bibliography{biblio}

\end{document}